\documentclass[10pt]{article}
\usepackage{geometry}
\usepackage{graphicx}
\usepackage{amsmath,amssymb,amsthm}
\usepackage[utf8]{inputenc}

\usepackage[greek,english]{babel}
\usepackage{enumerate}
\usepackage{mathrsfs}
\usepackage{hyperref}
\usepackage[active]{srcltx}

\numberwithin{equation}{section}
\usepackage[T1]{fontenc}
\usepackage{color}
\newtheorem{theo}{Theorem}
\newtheorem{lem}{Lemma}[section]
\newtheorem{defi}{Definition}[section]

\newtheorem{prop}{Proposition}[section]
\newtheorem{rmk}{Remark}[section]

\newcommand{\R}{\mathbb{R}}
\newcommand{\T}{\mathbb{T}}
\newcommand{\Z}{\mathbb{Z}}
\newcommand{\N}{\mathbb{N}}

\newcommand{\be}{\begin{equation}}
\newcommand{\ee}{\end{equation}}
\newcommand{\baa}{\begin{array}}
\newcommand{\eaa}{\end{array}}
\newcommand{\ba}{\begin{eqnarray}}
\newcommand{\ea}{\end{eqnarray}}

\newcommand{\bqs}{\begin{equation*}}
\newcommand{\eqs}{\end{equation*}}
\newcommand{\bqq}{\begin{equation}}
\newcommand{\eqq}{\end{equation}}

\DeclareMathAlphabet\mathbfcal{OMS}{cmsy}{b}{n}

\date{}

\begin{document}

\title{Continuity of pulsating wave speeds for bistable reaction-diffusion equations in spatially periodic media\thanks{This work has received funding from the National Natural Science Foundation of China (12001206) and the Basic and Applied Basic Research Foundation of Guangdong Province (2019A1515110506).}}

\author{Weiwei Ding\footnote{School of Mathematical Sciences, South China Normal University, Guangzhou 510631, China (dingweiwei@m.scnu.edu.cn)} \qquad Zhanghua Liang\footnote{School of Mathematical Sciences, South China Normal University, Guangzhou 510631, China}  \qquad Wenfeng Liu\footnote{School of Mathematical Sciences, South China Normal University, Guangzhou 510631, China} }
\maketitle

\begin{abstract}
This paper is concerned with pulsating waves for multi-dimensional reaction-diffusion equations in spatially periodic media. First, assuming the existence of pulsating waves connecting two linearly stable steady states, we study the continuity of wave speeds with respect to the direction of propagation. The continuity was proved in \cite{guo} under the extra condition that the speeds are nonzero in all directions. Here, we revisit this continuity result without the extra condition. Secondly, we provide some sufficient conditions ensuring the existence of pulsating waves in rapidly oscillating media, which allow the equations to have multiple stable steady states.
\end{abstract}

\section{Introduction and main results}

In this paper, we consider a spatially periodic reaction-diffusion equation of the form
\begin{equation}\label{eq:main}
\partial_t u  = \nabla \cdot (A(x)\nabla u) + f(x,u), \quad t \in \R \, , \ x \in \R^d.
\end{equation}
Here, the matrix field $A(x)=(A_{i,j}(x))_{1\leq i,j\leq d}$ is symmetric of class $C^{1,\alpha}$ for some $\alpha>0$,  and there exists $\alpha_1>0$ such that
\begin{equation}\label{uni-elliptic}
  \sum_{i,j} A_{i,j}(x)y_i y_j  \geq  \alpha_1 |y|^2  \,\,\hbox{ for all }  (x,y)\in\R^d\times\R^d.
 \end{equation}
The nonlinearity $(x,u)\in \R^d\times\R\mapsto f(x,u)$ is of class $C^1$, and $u\mapsto \partial_u f(x,u)$ is locally Lipschitz-continuous uniformly with respect to $x\in\R^d$.  Let  $L_1, L_2,\cdots,L_d$ be positive numbders and $L=(L_1,\cdots,L_d)$ be a vector in $\R^d$. We assume that $A(x)$ and $f(x,u)$ are $L$-periodic with respect to $x\in\R^d$ in the sense that
\begin{equation}\label{period}
A (x + k L ) \equiv A (x),\quad  f (x + k L ,\cdot) \equiv f (x, \cdot)   \,\,\hbox{ for any }\,\, k=(k_1,k_2,\cdots,k_d)\in \Z^d,
\end{equation}
where $kL:= (k_1 L_1 , \cdots , k_d L_d)$. We also assume that 

\begin{itemize}
\item [{\bf (A1)}] $f(\cdot,0)=f(\cdot,1)=0$ and the constants $0$ and $1$ are two linearly stable steady states of \eqref{eq:main} with respect to spatially $L$-periodic solutions.
\end{itemize}
Here, an $L$-periodic steady state $\bar{u}$ of \eqref{eq:main} is said to be linearly stable ({\it resp.} linearly unstable) if $\lambda(\bar{u})>0$ ({\it resp.} $\lambda(\bar{u})<0$), where $\lambda_1(\bar{u})$ is the principal eigenvalue of 
\begin{equation*}
-\nabla \cdot (A\nabla\psi)-\partial_uf(x,\bar{u})\psi=\lambda\psi\hbox{ in }\R^d,\quad \psi>0\hbox{ in }\R^d,\quad \psi\hbox{ is }L\hbox{-periodic}.
\end{equation*}

This work is concerned with the pulsating waves of \eqref{eq:main} connecting the two linearly stable steady states $0$ and $1$. Let us first recall the definition of pulsating waves.  For any $e\in \mathbb{S}^{d-1}$,  a \textit{pulsating wave} connecting 0 and 1 in the direction $e$ is an entire solution $U(t,x)$ of \eqref{eq:main} of the type
\begin{equation}\label{formula-tw}
U (t,x) = \Phi ( x \cdot e - c^*(e)t ,x),
\end{equation}
where $c^*(e) \in\R$ is the wave speed and the wave profile $\Phi$ is $L$-periodic in its second variable, and satisfies
$$\Phi (- \infty , \cdot) = 1 , \quad  \Phi ( + \infty , \cdot) = 0,$$
where both convergences are uniform with respect to the second variable.
If $c^*(e)\neq 0$, then the change of variable $(t,x)\mapsto (x \cdot e - c^*(e)t ,x)$ is revertible and $\Phi$ is uniquely determined by $U$. Moreover, by the periodicity of $\Phi$, there holds
\begin{equation*}
U\left(t+ \frac{kL\cdot e}{c^*(e)},x+kL\right)=U\left(t,x\right) \,\,\hbox{ for all }\,\, (t,x)\in\R\times\R^d, \,\,k\in \Z^d.
\end{equation*}
In this case, we call $U(t,x)$ is a {\it non-stationary pulsating wave}.  On the other hand, if $c^*(e)= 0$, a pulsating wave simply means a stationary solution $U(x)$ which solves \eqref{eq:main} and satisfies that  $U (x) \to 1$ ({\it resp.}~0) as $x \cdot e \to -\infty$ ({\it resp.} $+\infty$). In this case, $U(x)$ is called a {\it stationary pulsating wave}.

In population dynamics, pulsating wave is an important notion to understand the evolution process of the invasive species in spatially periodic media. It is a natural generation of the classical notion of traveling wave in spatially homogeneous media where the coefficients $A$ and $f$ are independent of $x$. 
In such a situation, a typical example for the existence of traveling wave connecting two stable zeros of $f$ (here, the zeros are $0$ and $1$) is that $f$ is of the {\it bistable type}, that is,  $f(0)=f(\theta)=f(1)=0$ for some $\theta\in(0,1)$, $f<0$ on $(0,\theta)$, $f>0$ on $(\theta,1)$ and $f'(0)<0$, $f'(1)<0$.  
For this homogeneous equation 
\begin{equation*}
u_t  = A\Delta u + f(u), \quad t \in \R \, , \ x \in \R^d,
\end{equation*}
it is well known (see \cite{aw,fm1}) that there exists a unique (up to shifts) traveling wave $u(t,x)=\phi(x\cdot e-ct)$ with $0<\phi<1$ in $\R$, $\phi (-\infty)=1$ and $\phi (+\infty)=0$ for any $e \in \mathbb{S}^{d-1}$. Moreover, the speed and the front are independent of $e$. Let us also mention that there may exist traveling waves connecting $0$ and $1$ when $f$ is of the {\it multistable type}, that is, it has multiple stable zeros between $0$ and $1$ (see \cite{fm1} and Proposition \ref{condi} below).

In the spatially periodic media, the existence of pulsating fronts of \eqref{eq:main} connecting $0$ and $1$ was first studied by Xin \cite{xin,xin2} in the case where $f$ is independent of $x$ and is of the unbalanced (i.e. $\int_0^1f(u)du\neq 0$) bistable type. The author obtained the non-stationary pulsating wave under the condition that $A(x)$ is sufficiently close to its integral average,  and also gave an example for the existence of stationary wave  when this condition is invalid. For general $x$-dependent $f$, pulsating waves are known to exist provided that equation \eqref{eq:main} admits a {\it bistable structure} in the sense that, apart from (A1), any $L$-periodic steady states strictly between $0$ and $1$ are required to be linearly unstable; see \cite{dgm,fz} for the one-dimensional results and \cite{ducrot,gr} for the multi-dimensional results. Moreover, explicit conditions on $A$ and $f$ ensuring the bistable structure can be found in \cite{dhz1,ducrot,fz}.  We also refer to \cite{m,nr,z} and references therein for more existence and non-existence results when $f$ is of the bistable type.  It is worthy noting that, the bistable structure is only a sufficient condition for the existence of pulsating waves, and
the existence may remain valid when \eqref{eq:main} has multiple periodic stable steady states between $0$ and $1$ (see the comments following Proposition \ref{condi}). 

Notice that when $A(x)$ or $f(x,u)$ truly depends on $x$, equation \eqref{eq:main} is not invariant by rotation, and therefore it is a natural to ask how the wave speed $c^*(e)$ depends on the direction $e$. Recently, it is shown in \cite{dg} that $c^*(e)$ can be dramatically different in distinct directions. But, it is still unclear whether $c^*(e)$ varies continuously with respect to $e$. One of our aims in the present  paper is to address this problem under the following assumption:

\begin{itemize}
\item [{\bf (A2)}] For any $e\in \mathbb{S}^{d-1}$, \eqref{eq:main} admits a pulsating wave connecting $0$ and $1$ with speed $c^*(e)\in\R$.
\end{itemize}

The problem on the continuity of wave speeds with respect to directions has been studied by Guo \cite{guo} under the additional condition that $c^*(e)\neq 0$ for all $e\in \mathbb{S}^{d-1}$ (i.e. the pulsating waves are non-stationary in all directions).  By some comparison arguments and compactness arguments, the author proved  the continuity of the wave speeds $e\mapsto c^*(e)$ as well as the continuity of wave profiles $e\mapsto \Phi(\xi,y)$. Here, we will remove the additional condition and revisit the continuity of wave speeds. It should be pointed out that, the case with stationary pulsating waves is degenerate, due to the lack of uniqueness of stationary waves (see \cite[Theorem 1.7]{dhz2}), and therefore without the additional condition, it may be not possible to pursue the continuity of wave profiles. We also mention that, for equation \eqref{eq:main} with monostable or ignition nonlinearities, the (minimal) wave speed depends continuously on the direction \cite{ag}. In those cases, the pulsating waves are non-stationary in all directions, and the (minimal) wave speed is  bounded from below by a positive constant uniformly with respect to $e\in\mathbb{S}^{d-1}$. Hence, the problem in our bistable case (under the conditions (A1) and (A2)) has extra difficulties.

Before stating the main results, let us first collect some qualitative properties of the pulsating waves which will be useful later. 

\begin{theo}\label{th:tw}
Let ${\rm (A1)}$ and ${\rm (A2)}$ hold. Then the speed $c^*(e)$ of pulsating waves for \eqref{eq:main} is unique and  the map $e\mapsto c^*(e)$ is bounded in $\mathbb{S}^{d-1}$. Furthermore, if $c^*(e)\neq 0$, then the following statements hold true:
\begin{itemize}
\item [{\rm (i)}] $c^* (e)$ has the sign of $\left(\int_0^1\int_{(0,L_1)\times\cdots\times(0,L_d)} f(x,u)dxdu\right)$;
\item [{\rm (ii)}] the wave $U(t,x)$ is increasing $(${\it resp.} decreasing$)$ with respect to $t\in\R$ if  $c^* (e)>0$  $(${\it resp.} if $c^* (e)<0$$)$;
\item [{\rm (iii)}] the wave $U(t,x)$ is unique up to shifts in time.
\end{itemize}
\end{theo}

This theorem follows from some known results and techniques in several earlier papers. First of all, the sign property is proved by a simple integration argument on the equation satisfied by the wave profile $\Phi$ (we refer to \cite{dhz1,ducrot} for the details). Next, the monotonicity and uniqueness of pulsating waves and the uniqueness of wave speeds can be shown by a standard sliding argument. Indeed,  these properties are easy consequences of  \cite[Theorems 1.11, 1.12 and 1.14]{bh2} provided that, in addition to (A1) and (A2), there exists some small $\delta>0$ such that the function $u\mapsto f(x,u)$ is nondecreasing in $u\in [0,\delta]\cup[1-\delta,1]$. Without this extra condition, these properties remain valid and the verification is almost parallel to that in the time-periodic case (see \cite[Propositions 3.4 and 3.6]{contri}). We mention that the uniqueness of wave speeds also follows easily from the comparison arguments used in the present paper (see Remark \ref{speed-unuqe} below). To obtain the boundedness of the wave speed $(c^*(e))_{e\in\mathbb{S}^{d-1}}$, one can use the same arguments as in the proof of \cite[Lemma 2.7]{dg} with some obvious modifications. Lastly, it is worthy noting that the sign property and the uniqueness of pulsating waves are only valid for non-stationary waves; see counterexamples for stationary waves in \cite{dhz2,xin2}.

Our first main theorem is stated as follows.

\begin{theo} \label{Continuity of speeds}
Let ${\rm (A1)}$ and ${\rm (A2)}$ hold.  Then the function $e\in \mathbb{S}^{d-1}\mapsto c^*(e)$ is continuous.
\end{theo}

The proof is inspired by \cite[Lemma 2.5]{dg} which is concerned with the continuity of pulsating wave speeds with respect to small perturbations on the nonlinearity $f$. It relies on the exponential decay rates when the pulsating waves approach the limiting steady states $0$ and $1$.  However, some parts of \cite[Lemma 2.5]{dg} still require the pulsating waves to be non-stationary. In our proof, we remove this condition completely by some new comparison arguments. Moreover, the same strategy can be applied to improve \cite[Lemma 2.5]{dg}, as remarked below precisely.

\begin{rmk} 
Let ${\rm (A1)}$ and ${\rm (A2)}$ hold.  Let  $(A_n)_{n\in\N} \subset C^{1,\alpha}(\R^d)$ be a sequence of $L$-periodic symmetric matrix satisfying \eqref{uni-elliptic} and let  $(f_n)_{n\in\N} \subset C^1(\R^{d+1})$ be a sequence of functions  such that each $f_n$ is $L$-periodic in the first variable and satisfies ${\rm (A1)}$.  Assume that
$A_n \to  A$ in $C^{1,\alpha}(\R^d)$ and   $f_n\to f$ in $C^1(\R^{d+1})$ as $n\to \infty$, and that for each $n\in\N$ and $e\in\mathbb{S}^{d-1}$, the equation
$$\partial_t u  = \nabla \cdot (A_n(x)\nabla u) + f_n(x,u), \quad t \in \R \, , \ x \in \R^d,$$
admits a pulsating wave connecting $0$ and $1$ with speed $c^*_n(e)\in\R$.
 Then $c^*_n(e)\to c^*(e)$ as $n\to +\infty$.
\end{rmk}

Next, we present the result on the existence of pulsating waves in rapidly oscillating environments. More precisely, we consider the following equation  
\begin{equation}\label{eql}
\partial_t u  = \nabla \cdot (A(x/l)\nabla u) + f(x/l,u), \quad t \in \R \, , \ x \in \R^d,
\end{equation}
with $l>0$. It is clear that $A(x/l)$ and  $f(x/l,\cdot)$ are $lL$-periodic in $x\in\R^d$ with  
$$lL:=(lL_1,\cdots,lL_d).$$ 
We show here that, for any direction $e\in \mathbb{S}^{d-1}$,  \eqref{eql} admits a pulsating wave connecting $0$ and $1$ when $l$ is small. To do so, we need the assumption that $0$ and $1$ are uniformly (in $x$) stable zeros of $f(x,\cdot)$ in the following sense:
\begin{itemize}
\item [{\bf (A3)}] $f(\cdot,0)=f(\cdot,1)=0$ and there exists $\gamma_0>0$ such that 
\begin{equation*}
\partial_uf(x,0)\le-\gamma_0  \quad \hbox{and} \quad \partial_uf(x,1)\le-\gamma_0 \,\,\,\hbox{ for all }\,\,\,x\in\R^d.
\end{equation*}
\end{itemize}
It is easily seen that (A3) is stronger than (A1), and it implies in particular that $0$ and $1$ are two linearly stable $lL$-periodic steady states of \eqref{eql} for all $l>0$.

To introduce our existence result, we need to set some notations. 
Denote by $\T^d:=\R^d/(L\Z^d)$ the $d$-dimensional torus, and denote by $\bar{f}:\R\to\R$ the integral average of the function $f$ with respect to the first variable over $\T^d$, that is, 
\begin{equation}\label{average} 
\bar{f}(u)= -\kern-10pt\int_{\T^d}f(x,u)dx\,\,\hbox{ for }\,\,u\in\R. 
\end{equation}
Clearly, $0$ and $1$ are two zeros of $\bar{f}$, and it follows from (A3) that $\max\{\bar{f}'(0),\bar{f}'(1)\}\leq -\gamma_0$. 
Let $A_0$ be a constant defined by
\begin{equation}\label{defA_0}
A_0 = -\kern-10pt\int_{\T^d}eA(x)(\nabla\chi(x)+e)dx,
\end{equation}
where $\chi(\cdot)\in C^2(\R^d)$ is the unique (up to a constant) solution of the equation 
\begin{equation}\label{de-chi}
\nabla \cdot(A(x)(\nabla \chi(x)+e))=0 \,\,\hbox{ in }\,\, \R^d,  \quad  \chi(x) \hbox{ is $L$-periodic in } \R^d.
\end{equation}
Notice that $A_0$ is a positive constant depending on the direction $e$. \footnote{The positivity of $A_0$ follows from the uniform ellipticity of $A(x)$. Indeed, multiplying the first equation of \eqref{de-chi} by $\chi(x)$ and integrating by parts over $\T^d$ gives $-\kern-7pt\int_{\T^d} \nabla \chi(x) A(x)(\nabla \chi(x)+e)dx=0$. This together with the definition of $A_0$ \eqref{defA_0} implies that $A_0=-\kern-7pt\int_{\T^d} (\nabla\chi(x)+e)A(x)(\nabla\chi(x)+e)dx$, which is positive due to \eqref{uni-elliptic}. }  Our theorem below shows that the existence of non-stationary pulsating wave of \eqref{eql} when $l$ is small is determined by the the existence of non-stationary traveling wave of the following homogenous problem
\begin{equation}\label{onehom}
v_t=A_0v_{zz}+\bar{f}(v),\ t\in\R,\ z\in\R. 
\end{equation}

\begin{theo}\label{thhomo}
Let {\rm (A3)} hold and let $e\in \mathbb{S}^{d-1}$ be arbitrary. Assume that equation \eqref{onehom} admits a traveling wave $v(t,z)=\phi_0(z-c_0t)$ with $c_0\neq 0$ and the limiting conditions $\phi_0(-\infty)=1$ and $\phi_0(+\infty)=0$. Then there is $l^*:=l^*(e)>0$ such that for any $0<l < l^*$, equation \eqref{eql} admits a unique (up to shifts in time) pulsating wave 
\begin{equation}\label{exist-small}
u_l(t,x)=\phi_l(x\cdot e-c^*_lt,x/l)
\end{equation}
connecting $0$ and $1$ with speed $c^*_l:=c^*_l(e)\neq 0$, where $\phi_l$ is $L$-periodic in the second variable. Furthermore, there holds $c^*_l\to c_0$ and up to translation of $\phi_0$, $\phi_l(\xi,y)-\phi_0(\xi)\to 0$ in $H^1(\R\times [0,L_1]\times\cdots \times [0,L_d])$ as $l \to0^+$.
\end{theo}

This theorem extends the  one-dimensional existence result \cite[Theorem 1.2]{dhz1} to the multi-dimensional case, and the proof shares the similarities, which replies on the application of the implicit function theorem in some appropriate Banach spaces. Similar idea was also appeared in an earlier paper \cite{he2} which is concerned with spatially homogeneous equations in periodically perforated domains.

Notice that the quantity $l^*$ provided by Theorem \ref{thhomo} may depend on the direction of propagation. It is unclear whether $l^*$ can be improved to be uniform in $e\in \mathbb{S}^{d-1}$. 
In the case $d=2$, for the special equation 
\begin{equation}\label{eq-d2}
\partial_t u  = \Delta u +u(u-a(x/l))(1-u) , \quad t \in \R \, , \ x \in \R^2,
\end{equation}
where $0<a(x)<1$ is $L$-periodic in $x\in\R^2$, a uniform $l^*$ is obtained in \cite[Theorem 1.10]{ducrot}. This is reached by showing that equation \eqref{eq-d2} admits a bistable structure when $l$ is small. But, the problem on the sign of the wave speed $c_l^*(e)$ remains open (see \cite[Remark 1.11]{ducrot}). Thanks to Theorem \ref{thhomo}, we can conclude  that 
\begin{equation}\label{sign-a}
sign(c_l^*(e))= sign\left(\frac{1}{2}--\kern-10pt\int_{\T^d}a(x)dx\right) \,\,\hbox{ for all }\,\, e\in\mathbb{S}^{d-1},\,\, 0<l<l^*.   
\end{equation} 
Indeed, if $-\kern-9pt\int_{\T^d}a(x)dx\neq 1/2$, then it is well known (see e.g. \cite{fm1}) that $c_0$ has the sign of $1/2--\kern-9pt\int_{\T^d}a(x)dx$, 
where $c_0$ is the traveling wave speed of \eqref{onehom} with $A_0=1$ and $\bar{f}(u)=u(u--\kern-9pt\int_{\T^d}a(x)dx)(1-u)$. This together with the convergence of $c_l^*(e)$ as $l\to 0^+$ stated in Theorem \ref{thhomo} immediately implies \eqref{sign-a}.  On the other hand, if $-\kern-9pt\int_{\T^d}a(x)dx= 1/2$, then $\int_0^1\int_{(0,L_1)\times\cdots\times(0,L_d)} u(u-a(x))(1-u)dxdu=0$. It follows from Theorem \ref{th:tw} (i) that $c_l^*(e)=0$, and hence, \eqref{sign-a} remains valid. 

The assumption on the existence of non-stationary traveling wave of \eqref{onehom}  plays an essential role in showing Theorem \ref{thhomo}. Below, we give some sufficient conditions ensuring the existence of this homogeneous wave. Setting 
\begin{equation}\label{de-F}
F(u):=\int_0^u\bar{f}(s)ds \,\, \hbox{ for }\,\,  u\in [0,1],
\end{equation}
we have the following proposition. 
\begin{prop}\label{condi}
Assume that $\bar{f}(0)=\bar{f}(1)=0$. Then, equation \eqref{onehom} admits a traveling wave connecting $0$ and $1$ with speed $c_0\neq0$ provided that one of the following conditions holds:
\begin{itemize}
\item [{\rm (a)}] $F(1)<0$ and $\bar{f}(u)< 0$ for all $u\in \left\{u\in (0,1):\, F(u)> F(1)\right\}$;
\item [{\rm (b)}] $F(1)>0$ and $\bar{f}(u)> 0$ for all  $u\in \left\{u\in (0,1): \,F(u)> F(0)\right\}$. 
\end{itemize}
Furthermore, $c_0>0$ $($resp. $c_0<0$$)$ if {\rm (b)} $($resp. {\rm (a)}$)$ holds true.  
\end{prop}

Notice that the above condition does not require $\bar{f}'(0)<0$ and $\bar{f}'(1)<0$. Indeed, it covers a wild class of nonlinearities that include not only standard monostable, bistable and combustion types but also multistable type, and even the nonlinearities with infinite many zeros between $0$ and $1$. This in particular implies that Theorem \ref{thhomo} allows equation \eqref{eql} to have multiple stable steady states between $0$ and $1$.  For instance, let the nonlinearity  $f$ be independent of $x$, and have the form
$$f(u)=u\left(u-\frac{1}{8}\right) \left(u-\frac{1}{4}\right)\left(u-\frac{3}{8}\right)(1-u)\,\,\hbox{ for }\,\,u\in[0,1].$$
It is straightforward to check that $\bar{f}(u)=f(u)$ satisfies condition (b) of Proposition \ref{condi}, and hence by Theorem \ref{thhomo}, for each direction $e\in\mathbb{S}^{d-1}$, equation \eqref{eql} admits a unique (up to shifts in time) pulsating wave connecting $0$ and $1$ with speed $c^*_l(e)>0$ when $l$ is small. On the other hand, it is easily seen that $u\equiv 1/4$ is a stable steady state of \eqref{eql}.

\section{Continuity of wave speeds}

This section is devoted to the proof of Theorem \ref{Continuity of speeds} under the assumptions (A1) and (A2). The proof relies on a Fife-McLeod type sub- and super-solutions and the exponential decay of pulsating waves when they approach the stable limiting states. For clarity, we first show these properties in Section \ref{pre-properties}, and then carry out the main proof in Section \ref{main-proof}.

\subsection{Preliminaries on the properties of pulsating waves}\label{pre-properties}

Throughout this subsection, the direction $e\in \mathbb{S}^{d-1}$ is fixed, and we write $c^*$ instead of $c^*(e)$ for convenience, and denote by $U$ the pulsating wave of \eqref{eq:main} connecting $0$ and $1$ in the direction $e$.  We first construct a pair of  sub- and super-solutions of \eqref{eq:main} when $U=U(t,x)$ is not stationary.

Recall that $0$ and $1$ are two linearly stable steady states of \eqref{eq:main} with respect to spatially $L$-periodic solutions. Let $\lambda^+$ and $\lambda^-$ be, respectively, the principal eigenvalues of
\begin {equation}\label{eigenvalue1}
-\nabla \cdot (A\nabla\psi)-\partial_uf(x,1)\psi=\lambda\psi  \,\hbox{ in } \R^d,\quad  \psi>0 \,\hbox{ in }\, x\in\R^d,\quad \psi \,\hbox{ is $L$-periodic},
\end {equation}
and
\begin {equation}\label{eigenvalue0}
-\nabla \cdot (A\nabla\psi)-\partial_uf(x,0)\psi=\lambda\psi  \,\hbox{ in } \R^d,\quad  \psi>0 \,\hbox{ in }\, x\in\R^d,\quad \psi \,\hbox{ is $L$-periodic}.
\end {equation}
It is clear that $\lambda^{\pm}>0$. Denote by $\psi^\pm(x)$ the corresponding positive periodic eigenfunctions. Since  $\psi^\pm$ are unique up to multiplication by a positive constant, for definiteness, we normalize them by requiring
\begin{equation}\label{normalization}
\|\psi^+\|=1,\quad  \|\psi^-\|=1,
\end{equation}
where $\|\cdot\|$ denotes the $L^\infty$-norm in $C(\R^d)$. Let $\rho\in C^2(\R,[0,1])$ be a function satisfying
\begin{equation}\label{rho}
\rho(\xi)=0 \,\,\,\, \hbox{in}\, \,\,[2,\infty),\,\, \,\,\rho(\xi)=1\,\,\,\, \hbox{in}\,\, \,(-\infty,0],\,\,\,\, -1\leq \rho'(\xi)\leq0\,\,\,\, \hbox{and} \,\,\,\,|\rho''(\xi)|\leq1 \,\,\hbox{in} \,\,\R.
\end{equation}
Our sub- and super-solutions are stated as follows.

\begin{lem}\label{sub-super-nonzero}
Assume that $c^* \neq 0$.  Then, there exist $\varepsilon_0>0$,  $\mu>0$ and $K\in\R$ (which has the sign of $c^*$) such that for any $\varepsilon\in (0,\varepsilon_0] $ and $k\in\Z^d$, the functions $u^+(t,x)$ and  $u^-(t,x)$ defined by
\begin{equation*}
\left.\begin{array}{ll}
 u^\pm(t,x):= \!\!\!\!& U(t\pm\varepsilon K(1- e^{-\mu t}),x+ kL)\pm\varepsilon e^{-\mu t}(\rho (x\cdot e-c^*t)\psi^+(x) \vspace{5pt}\\
&\qquad +(1-\rho (x\cdot e-c^*t))\psi^-(x))   \quad  \hbox{in} \quad (0,\infty)\times\R^d,
\end{array}\right.
\end{equation*}
are, respectively, a super-solution and a sub-solution of \eqref{eq:main}.
\end{lem}

We show this lemma by using a Fife-McLeod \cite{fm1} type sub- and super-solutions method. Similar arguments were used in the study of bistable equations/systems in various heterogeneous media (see e.g., \cite{abc,dhz1,dl,xin2,zz21}). In particular, in the spatially periodic case, if $f(x,u)$ is independent of $x$ in a neighborhood of $u=0$ and $u=1$, then the proof follows directly from that of \cite[Lemma 2.2]{dg}. Below, with some modifications, we include the details in the general case.

\begin{proof}
 We only give the construction of the sub-solution, as the analysis for the super-solution is analogous.  Without loss of generality, we assume that $c^*>0$ (we will sketch below that the case where  $c^*<0$ can be treated similarly). Then by Theorem \ref{th:tw} (ii), $U$ is increasing in its first variable.

Let us first set some notations. Since the function $(x,u) \mapsto \partial_uf(x,u)$ is continuous, there exists a small constant $\delta_0\in (0,1/4)$ such that
\begin{equation}\label{supf}
\left\{
\begin{array}{l}
\displaystyle \sup\limits_{x\in\R^d,u\in(-2\delta_0,2\delta_0)}\,\,\,\,\,\,\,\,|\partial _u f(x,0)-\partial _u f(x,u)|\leq \frac {\lambda^-}{2},\vspace{5pt}\\
\displaystyle \sup\limits_{x\in\R^d,u\in(1-2\delta_0,1+2\delta_0)}|\partial _u f(x,1)-\partial _u f(x,u)|\leq \frac {\lambda^+}{2}.
\end{array}\right.
\end {equation}
Let $k\in\Z^d$ be arbitrary. By the definition of pulsating waves, there exists $M\geq 3+|kL|$ sufficiently large such that
\begin{equation*}
\left\{
\begin{array}{ll}
\displaystyle 0< U(t,x)<\delta_0 & \hbox{for all }\,  \, x\cdot e-c^*t\geq M, \vspace{5pt}\\
\displaystyle 1-\delta_0< U(t,x)<1 & \hbox{for all }\,  \,  x\cdot e-c^*t\leq -M,  \vspace{5pt}\\
\displaystyle \frac{\delta_0}{2}< U(t,x)<1-\frac{\delta_0}{2} & \hbox{for all }\,  \, -M<x\cdot e-c^*t< M.
\end{array}\right.
\end {equation*}
Set $\mu=\min\{\lambda^-/2,\lambda^+/2\}$, and
$$K=\frac{B_1+B_2+\mu \|\psi^++\psi^-\|}{\beta \mu},$$
where $B_1$, $B_2$ and $\beta$ are positive constants given by
$$\beta=\min_{\delta_0/2\leq U(t,x)\leq1-\delta_0/2} \partial_t U(t,x),$$
$$B_1=(|c^*|+\|eAe\|+\|\nabla \cdot(Ae)\|)\|\psi^+-\psi^-\|+2\|eA \nabla(\psi^+-\psi^-)\|+\|\nabla\cdot(A\nabla\psi^+)\|+\|\nabla\cdot(A\nabla\psi^-)\|,$$
and
$$B_2=\max_{x\in\R^d,u\in[-\delta_0,1+\delta_0]}|\partial _u f(x,u)|\,\|\psi^++\psi^-\|$$
(by standard elliptic estimates applied to \eqref{eigenvalue1} and \eqref{eigenvalue0}, it is easily seen that $B_1$ and $B_2$ are bounded constants).
Choose a small constant $\varepsilon_0\in (0,\delta_0)$ such that
\begin {equation}\label{varepsilon_0}
0<\varepsilon_0\leq \min\left\{ \frac{\delta_0}{\|\psi^++\psi^-\|},\, \frac{1}{c^*K}  \right\}.
\end {equation}

Next, define
$$\mathcal{L}u^-:=\partial_tu^--\nabla \cdot (A(x)\nabla u^-) -f(x,u^-)\,\,\hbox{ for }\,\,(t,x)\in (0,\infty)\times\R^d.$$
For the above $\mu>0,$ $K>0$ and $\varepsilon_0>0,$ we will show that for any $0<\varepsilon\leq \varepsilon_0,$ $\mathcal{L}u^-\leq0 $ in $(0,\infty)\times\R^d$. Let  $\varepsilon \in (0, \varepsilon_0)$ be arbitrary, and define
\begin{equation*}
q(t)=\varepsilon e^{-\mu t} \,\,\,\,\hbox{and}\,\,\,\, \eta (t)=\varepsilon K(1-e^{-\mu t})\,\,\,\, \hbox{for} \,\,\,\, t>0.
\end{equation*}
Since  $U(t,x)$ is an entire solution of \eqref{eq:main} and thanks to the spatial periodicity, a straightforward computation gives that, for all $(t,x)\in(0,\infty)\times\R^d$,
\begin{equation}\label{cal-lu-}
\begin{split}
\mathcal{L}u^-=-\eta'\partial_tU-q'(\rho \psi^++(1-\rho)\psi^-)+c^*q\rho'(\psi^+-\psi^-)-(f(x,u^-)-f(x,U))\vspace{5pt}\\
+ q[2\rho'eA(\nabla(\psi^+-\psi^-))+(\rho''eAe+\rho'\nabla\cdot(Ae))(\psi^+-\psi^-)\vspace{5pt}\\
+ \rho\nabla \cdot (A\nabla\psi^+ )+(1-\rho)\nabla \cdot (A\nabla\psi^- )]
,\end{split}
\end{equation}
where $\rho$, $\rho'$ and $\rho''$ are evaluated at $x\cdot e-c^*t$, $q$, $q'$, $\eta$ and $\eta'$ are evaluated at $t$, and $\psi^\pm$, $A$ are evaluated at $x$. Below, we complete the proof by considering three cases.

Case (a): $x\cdot e+kL\cdot e-c^*(t-\eta(t))\leq-M$.

By \eqref{varepsilon_0}, we have $0<\varepsilon c^*K\leq1$, whence $0\leq c^*\eta(t)\leq1$ and $x\cdot e-c^*t\leq-M+|kL|<0$. Then, by \eqref{rho}, we have $\rho=1$, $\rho'=\rho''=0$.
It follows from \eqref{cal-lu-} that
$$\mathcal{L}u^-=-\eta'\partial_tU-q'\psi^++q\psi^+\partial_u f(x,U-\theta_1q\psi^+)+q\nabla \cdot (A\nabla\psi^+ )$$
for some $\theta_1=\theta_1(t,x)\in[0,1]$. Remember that $U$ is increasing in its first variable. This together with $\eta'(t)>0$ and the fact that $\lambda^+$ is the principal eigenvalue of \eqref{eigenvalue1} implies that
$$ \mathcal{L}u^-\leq-q'\psi^+-\lambda^+q\psi^+-q\psi^+[\partial_uf(x,1)-\partial_u f(x,U-\theta_1q\psi^+)]. $$
Since $U-\theta_1q\psi^+\in (1-2\delta_0,1)$, by the second inequality of \eqref{supf}, it follows that
$$ \mathcal{L}u^-\leq-q'\psi^+-\lambda^+q\psi^++\frac{\lambda^+}{2}q\psi^+.$$
Since $q(t)=\varepsilon e^{-\mu t}$ with  $0<\mu\leq \lambda^+/2$, it is then easily checked that $\mathcal{L}u^-\leq0$.

Case (b): $x\cdot e+kL\cdot e-c^*(t-\eta(t))\geq M$.

In this case, we have $0\leq c^*\eta(t)\leq1$, $x\cdot e-c^*t\geq M-c^*\eta(t)-|kL|>2$, and hence, $\rho=\rho'=\rho''=0$. It follows that
\begin{equation*}
\left.\begin{array}{ll}
 \mathcal{L}u^-\!\!\!&=-\eta'\partial_tU-q'\psi^-+q\psi^-\partial_u f(x,U-\theta_2q\psi^-)+q\nabla \cdot (A\nabla\psi^-) \vspace{5pt}\\
 & \leq-q'\psi^--\lambda^-q\psi^--q\psi^-[\partial_uf(x,0)-\partial_u f(x,U-\theta_2q\psi^-)]
\end{array}\right.
\end{equation*}
for some $\theta_2=\theta_2(t,x)\in[0,1]$, where $\lambda^-$ is the principal eigenvalue of \eqref{eigenvalue0}.
Now, proceeding similarly as in Case (a), by using this time the first inequality of \eqref{supf} and the fact that $0<\mu\leq\lambda^-/2$, we obtain that $\mathcal{L}u^-\leq0$ when $x\cdot e+kL\cdot e-c^*(t-\eta(t))\geq M$.

Case (c): $-M<x\cdot e+kL\cdot e-c^*(t-\eta(t))<M$.

In this case, we have $U(t-\eta(t),x+ kL)\in [\delta_0/2,1-\delta_0/2]$. This together with \eqref{varepsilon_0} implies that $u^-(t,x)\in [-\delta_0/2,1-\delta_0/2]$. Remember that $\rho\in C^2(\R,[0,1])$ satisfies \eqref{rho}. By the above choice of $B_1$, $B_2$,
it is straightforward to check that
\begin{equation*}
\left.\begin{array}{ll}
c^*q\rho'(\psi^+-\psi^-)+q[2\rho'eA(\nabla(\psi^+-\psi^-)\!\!\!\!&+(\rho''eAe  +\rho'\nabla\cdot(Ae))(\psi^+-\psi^-) \vspace{5pt}\\
& +\rho\nabla \cdot (A\nabla\psi^+ )+(1-\rho)\nabla \cdot (A\nabla\psi^- )]\leq B_1q,
\end{array}\right.
\end{equation*}
and
$$f(x,U)-f(x,u^-)=q(\rho \psi^++(1-\rho)\psi^-)\partial_u f(x,U-\theta_3q(\rho \psi^++(1-\rho)\psi^-))\leq B_2q,$$
for some $\theta_3=\theta_3(t,x)\in[0,1]$. Furthermore, with the choice of $\beta$ and $K$, we compute that
\begin{equation*}
\left.\begin{array}{ll}
\mathcal{L}u^- \!\!\!&\leq-\beta\eta'-q'\|\psi^++\psi^-\|+B_1q+B_2q \vspace{5pt}\\
&=-\mu K\beta q+\mu\|\psi^++\psi^-\|q+B_1q+B_2q=0.
\end{array}\right.
\end{equation*}

Combining the above, we obtain that $\mathcal{L}u^-\leq0 $ in $(0,\infty)\times\R^d$. Thus, $u^-$ is a sub-solution of \eqref{eq:main}. Finally, in the case where $c^*<0$, notice that $U$ is decreasing in its first variable (see Theorem \ref{th:tw} (ii)). It is easily checked that the above arguments remain valid with
$\beta$ and $K$ by replaced by $\min_{\delta_0/2\leq U(t,x)\leq1-\delta_0/2} \{-\partial_t U(t,x)\}$ and $-K$, respectively. As a consequence, there exists $K\in\R$ which has the sign of $c^*$ such that $u^-(t,x)$ is a sub-solution of \eqref{eq:main}. The proof of Lemma \ref{sub-super-nonzero} is thus complete.
\end{proof}

\begin{rmk}\label{speed-unuqe}
We point out that Lemma \ref{sub-super-nonzero} immediately implies the uniqueness of bistable wave speeds. More precisely, if $U_1$ and $U_2$ are two pulsating waves of \eqref{eq:main} connecting $0$ and $1$ in the direction $e$ with speeds $c_1^*$ and $c_2^*$, respectively, then by similar arguments to those used in the proof of \cite[Lemma 1.2]{dg}, one can conclude that $c^*_1=c^*_2$.
\end{rmk}

Next, we turn to show the exponential decay of the pulsating wave $U$ when it approaches the limiting states $0$ and $1$. For clarity, we divide the proof into two parts, according to whether $U$ is stationary or not.

Let us first set some notations. For any $\mu\in\R$, let $\lambda_1^+(\mu)$ and $\lambda_1^-(\mu)$ be, respectively, the principal eigenvalues of
\begin{equation*}
\left\{
\begin{array}{c}
-\nabla \cdot(A\nabla\psi)-2\mu(eA\nabla\psi)-\left(\mu^2eAe+\mu\nabla\cdot(Ae)+\partial_uf(x,1)\right)\psi =\lambda \psi \, \hbox{ in } \,\R^d ,\vspace{5pt}\\
\psi \  \hbox{is $L$-periodic},\quad \psi>0 \, \hbox{ in } \,\R^d,
\end{array}\right.
\end{equation*}
and
\begin{equation}\label{p-lambda-}
\left\{
\begin{array}{c}
-\nabla \cdot(A\nabla\psi)-2\mu(eA\nabla\psi)-\left(\mu^2eAe+\mu\nabla\cdot(Ae)+\partial_uf(x,0)\right)\psi =\lambda \psi \, \hbox{ in } \,\R^d ,\vspace{5pt}\\
\psi \  \hbox{is $L$-periodic},\quad \psi>0 \, \hbox{ in } \,\R^d.
\end{array}\right.
\end{equation}
It is clear that $\lambda_1^{\pm}(0)=\lambda^{\pm}$, where $\lambda^{\pm}$ are the principal eigenvalues of \eqref{eigenvalue1}-\eqref{eigenvalue0}, and hence, $\lambda_1^{\pm}(0)>0$. On the other hand,
since there exist constants $\beta_1>0$ (depending on $A$)  and $\beta_2>0$ (depending on $A$ and $f$) such that
$$\min\left\{\mu^2eAe+\mu\nabla\cdot(Ae)+\partial_uf(x,0),\,\,\mu^2eAe+\mu\nabla\cdot(Ae)+\partial_uf(x,1)\right\}\geq \beta_1\mu^2-\beta_2$$
for all $x\in\R^d$, $\mu\in\R$, it follows that $\lambda_1^\pm(\mu)\leq-\beta_1\mu^2+\beta_2$ for all $\mu\in\R$. As a consequence, $\lambda_1^\pm(\mu)\to-\infty$ as $|\mu|\to \infty$. Combining the above with the fact that
the functions $\mu\mapsto \lambda_1^\pm(\mu)$ are continuous and concave (see e.g., \cite[Proposition 5.7]{bh1}),
we can conclude that the equations
$$\lambda_1^\pm(\mu)=c^*\mu$$
have only two roots, and one is positive and the other one is negative,
where $c^*\in\R$ is the (unique) pulsating wave speed of \eqref{eq:main}. In the sequel, denote by $\underline{\mu}:=\underline{\mu}(c^*)$ the negative roof of $\lambda_1^-(\mu)=c^*\mu$ and by
$\bar{\mu}:=\bar{\mu}(c^*)$ the positive root of  $\lambda_1^+(\mu)=c^*\mu$.

The exponential decay of non-stationary pulsating waves is stated as follows.

\begin{lem}\label{exp-decay-hd-1}
Assume that the pulsating wave $U$ is non-stationary, that is, $U(t,x)=\Phi(x\cdot e-c^*t,x)$ with $c^*\neq 0$. Then,
\begin{equation}\label{exponen-j-hd}
\frac{\partial_{\xi}\Phi(\xi,x)}{\Phi(\xi,x)}\to \underline{\mu} \,\,\hbox{ as } \,\,\xi\to + \infty,\quad
\frac{\partial_{\xi}\Phi(\xi,x)}{1-\Phi(\xi,x)}\to -\bar{\mu} \,\,\hbox{ as } \,\,\xi\to - \infty,
\end{equation}
where all the convergences  hold uniformly in $x\in\R^d$. Moreover, for any small $\varepsilon>0$, there exist constants $C_+\geq C_->0$ such that
\begin{equation}\label{non-decay+}
C_- e^{(\underline{\mu}-\varepsilon)\xi} \leq   \Phi(\xi,x) \leq  C_+ e^{(\underline{\mu}+\varepsilon)\xi} \,\,\hbox{ for all }\,\,\xi\geq 0,\,\,x\in\R^d,
\end{equation}
and
\begin{equation}\label{non-decay-}
C_- e^{(\bar{\mu}+\varepsilon)\xi} \leq  1- \Phi(\xi,x) \leq  C_+ e^{(\bar{\mu}-\varepsilon)\xi} \,\,\hbox{ for all }\,\,\xi\leq 0,\,\,x\in\R^d.
\end{equation}
\end{lem}

\begin{proof}
The proof is similar to that of \cite[Lemma 2.3]{dg} with some minor modifications; therefore we only give the outline.

Let us first show the fist convergence stated in \eqref{exponen-j-hd}. Since $c^*\neq 0$ and since $\partial_\xi \Phi (\xi,x)=- \partial_t U ((x\cdot e-\xi)/c^*,x)/c^*$, standard parabolic estimates applied to equation \eqref{eq:main} imply that $\partial_\xi \Phi /\Phi$ is globally bounded in $\R\times\R^d$.
This together with the fact that $\partial_\xi \Phi (\xi,x) < 0$ for $(\xi,x)\in\R\times\R^d$ yields
$$\mu_*:=\limsup_{\xi\to + \infty}\sup_{x\in\R^d}\frac{\partial_\xi \Phi (\xi,x)}{\Phi(\xi,x)} \in (-\infty,0].$$
Since $\Phi(\xi,x)$ is $L$-periodic in $x$, there exists a sequence $(\xi_n,x_n)\in \R\times[0,L_1]\times\cdots\times[0,L_d]$ such that $\xi_n\to + \infty$ as $n\to +\infty$ and
$$\frac{\partial_\xi \Phi (\xi_n,x_n) }{\Phi(\xi_n,x_n) }\to \mu_* \,\hbox{ as } \, n\to +\infty.  $$
Up to extraction of some subsequence, we may assume that $x_n\to x_*$ as $n\to\infty$ for some
$x_*\in [0,L_1]\times\cdots [0,L_d]$.
For each $n\in\N$, set
$$t_n=\frac{x_n\cdot e-\xi_n}{c^*} \quad\hbox{and}\quad V_n(t,x)=\frac{U(t+t_n,x)}{U(t_n,x_n)}\, \, \hbox{ for }\, (t,x)\in\R\times\R^d.$$
By standard parabolic estimates, up to extraction of some subsequence, $V_n \to V_{\infty}$ in $C^{1,2}_{loc}(\R\times\R^d)$, where $V_{\infty}(t,x)$ is a positive solution of
\begin{equation}\label{eq-Vinfty-hd}
\partial_t V_{\infty}(t,x)=\nabla \cdot(A(x)\nabla  V_{\infty}(t,x)) +\partial_uf(x,0)V_{\infty}(t,x) \, \ \hbox{ for }  \,(t,x)\in\R\times\R^d.
\end{equation}

Furthermore, by the definition of $\mu_*$, one can check that for $(t,x)\in\R\times\R^d$,
$$W(t,x):=\frac{\partial_t V_{\infty}(t,x)}{V_{\infty}(t,x)}\geq -c^*\mu_*\geq 0 \,\,(\hbox{{\it resp. }}\leq -c^*\mu_*\leq 0)\,\,\hbox{ if }\,\, c^*>0\,\, (\hbox{{\it resp.}}\,\,\hbox{ if } \,\,c^*<0),$$
and  $W(t,x)$ obeys 
$$\partial_t W(t,x)=\nabla \cdot(A(x)\nabla W(t,x))+ 2\frac{\nabla V_{\infty}(t,x)}{V_{\infty}(t,x)}  A(x)\nabla W(t,x) \, \ \hbox{ for }  \,(t,x)\in\R\times\R^d.$$
Since $W(0,x_*)=-c^*\mu_*$, applying the strong maximum principle to the equation satisfied by $W(t,x)$, one obtains
 $W \equiv -c^*\mu_*$ in $\R\times\R^d$. This together with \eqref{eq-Vinfty-hd} implies that
$$V_{\infty}(t,x)=e^{\mu_* (x\cdot e-c^*t)}\psi(x)\, \hbox{ for } \,(t,x)\in\R\times\R^d, $$
where $\psi(x)$ is a positive $L$-periodic solution of
$$ -\nabla \cdot(A\nabla\psi)-2\mu_*(eA\nabla\psi)-\left(\mu_*^2eAe+\mu_*\nabla \cdot(Ae)+\partial_uf(x,0)\right)\psi=\mu_* c^*\psi,\,\,\, x\in\R^d.$$
In other words, $\mu_* c^*$ is the principal eigenvalue of \eqref{p-lambda-} with $\mu=\mu_*$. By the uniqueness of  principal eigenvalues,  we have $\lambda_1^-(\mu)=\mu c^*$. Since $\mu_*\leq 0$, there must hold $\mu_*=\underline{\mu}$, that is, $\mu_*$ is the negative root of $\lambda_1^-(\mu)=c^*\mu$. Now, we reach that $\limsup_{\xi\to +\infty}\sup_{x\in\R^d} \partial_\xi \Phi/\Phi=\underline{\mu}$.
In a similar way, one can prove that $\liminf_{\xi\to +\infty}\inf_{x\in\R^d}\partial_\xi \Phi /\Phi=\underline{\mu}$.
This ends the proof of the first convergence of \eqref{exponen-j-hd}.

Next, writing $V(t,x)=1-U(t,x)$ for $(t,x)\in\R\times\R^d$ and $\Psi(\xi,x)=1-\Phi(\xi,x)$ for $(\xi,x)\in\R\times\R^d$,
one checks that $V(t,x)=\Psi(x\cdot e-c^*t,x)$ is a pulsating wave of the following equation
\begin{equation*}
\partial_t v  = \nabla \cdot (A\nabla v) + g(x,v), \quad t \in \R \, , \ x \in \R^d,
\end{equation*}
with the limit conditions $\Psi(\xi,x)\to 1$ as $\xi \to +\infty$ and $\Psi(\xi,x)\to 0$ as $\xi \to -\infty$ uniformly in $x\in\R^d$, where $g$ is a function given by $g(t,v)=-f(x,1-v)$. It is clear that $\partial_vg(x,0)\equiv \partial_u f(x,1)$. Notice that $\Psi(\xi,x)$ is increasing in $\xi\in\R$, and hence, $\partial_\xi \Psi / \Psi \geq  0$ in $\R\times\R^d$.
Then, proceeding similarly as above, one can conclude that
$$ \frac{\partial_\xi \Psi (\xi,x)}{\Psi(\xi,x)} \to \bar{\mu} \,\,\hbox{ as }\,\, \xi\to-\infty\,\,\hbox{ uniformly in } \,x\in\R^d.$$
This immediately implies the second part of \eqref{exponen-j-hd}.

Finally, once \eqref{exponen-j-hd} is obtained, the estimates \eqref{non-decay+} and \eqref{non-decay-} follow directly from the discussion in \cite[Remark 2.1]{dg}.  The proof of Lemma \ref{exp-decay-hd-1} is thus complete.
\end{proof}

The following lemma provides the exponential decay of stationary pulsating waves.

\begin{lem}\label{exp-decay-hd-2}
Assume that the pulsating wave $U=U(x)$ is stationary, that is, $c^*=0$.  Then,  for any small 
$\varepsilon>0$, there exist constants $C_+\geq C_->0$ such that
\begin{equation}\label{s-decay+}
C_- e^{(\underline{\mu}-\varepsilon)x\cdot e} \leq   U(x) \leq  C_+ e^{(\underline{\mu}+\varepsilon)x\cdot e} \,\,\hbox{ for all }\,\,x\cdot e \geq 0,
\end{equation}
and
\begin{equation*}
C_- e^{(\bar{\mu}+\varepsilon)x\cdot e} \leq 1- U(x)  \leq  C_+ e^{(\bar{\mu}-\varepsilon)x\cdot e} \,\,\hbox{ for all }\,\,x\cdot e \leq 0.
\end{equation*}
\end{lem}

\begin{proof}
We only give the proof of \eqref{s-decay+}, since similar arguments applied to the equation satisfied by $1- U(x)$ give the estimates in the case $x\cdot e \leq 0$.

Remember that the function $\mu\in\R\mapsto \lambda_1^-(\mu)$ is concave and continuous, that $\lambda_1^-(0)>0$, and
that $\underline{\mu}$ is the negative root of $\lambda_1^-(\mu)=0$ (recall that $c^*=0$ in the present lemma). Then, there exists  some small $\varepsilon_1>0$ such that  $\lambda_1^-(\mu)$ is increasing in $[\underline{\mu}-\varepsilon_1,\underline{\mu}+\varepsilon_1]\subset (-\infty,0)$, and that
$$ \lambda_1^-(\underline{\mu}-\varepsilon)<0<  \lambda_1^-(\underline{\mu}+\varepsilon) <  \lambda_1^-(0) \,\,\hbox{ for any }\,\, \varepsilon \in (0, \varepsilon_1].$$
In the arguments below, $\varepsilon \in (0, \varepsilon_1]$ is arbitrary. Denote by
$$\mu_{\varepsilon,\pm}:=\underline{\mu}\pm \varepsilon,$$
and
by $\psi_{\varepsilon,\pm}(x)$ the positive principal eigenfunction of \eqref{p-lambda-} with $\mu$ replaced by $\mu_{\varepsilon,\pm}$, respectively. We normalize $\psi_{\varepsilon,\pm}(x)$ by requiring that $\|\psi_{\varepsilon,\pm}\|=1$.  Moreover,  since the function $u \mapsto \partial_uf(x,u)$ is continuous uniformly in $x\in\R^d$ and $f(x,0)=0$, there exists a small constant $\delta_0\in (0,1/2)$ such that
\begin{equation}\label{derivatives-u}
(\partial_uf(x,0)-\varepsilon')u\leq  f(x,u )\leq (\partial_uf(x,0)+\varepsilon')u\,\,\hbox{ for all }\,\, u\in [0,\delta_0],\,\,x\in\R^d,
\end{equation}
where $\varepsilon'=\min\left\{ \lambda_1^-(\mu_{\varepsilon,+}),\, -\lambda_1^-(\mu_{\varepsilon,-})\right\}$.
By the definition of stationary pulsating waves, there exists $\xi_+ \geq 0$ such that
\begin{equation}\label{find-delta0}
0<U(x) \leq \delta_0 \,\,\hbox{ for all }\,\, x\cdot e \geq \xi_+.
\end{equation}
It is easily seen that
\begin{equation}\label{length-xi+}
0< \inf_{x\cdot e\in [0,\,\xi_++1]} U(x) \leq  \sup_{x\cdot e\in [0,\,\xi_++1]} U(x) <1.
\end{equation}
Now, we define
$$u_{\pm}(x)=K_{\pm} e^{\mu_{\varepsilon,\pm} x\cdot e}  \psi_{\varepsilon,\pm}(x)\,\,\hbox{ for }\,\,x\in\R^d,$$
where $K_-\leq K_+$ are two positive constants satisfying
\begin{equation}\label{com-boundary}
u_-(x) \leq  U(x) \leq  u_+(x)  \,\,\hbox{ for all }\,\, x\in\R^d \,\,\hbox{ such that } x\cdot e \in [0,\xi_++1].
\end{equation}
Indeed, the existence of $K_{\pm}$ follows from \eqref{length-xi+} and the fact that $\psi_{\varepsilon,\pm}(x)$ can be bounded from above and below by positive constants in $\R^d$.  It is straightforward to check that  $u_{\pm}(x)$ satisfy
$$ -\nabla \cdot (A\nabla u_{\pm}) -(\partial_uf(x,0)+ \lambda_1^-(\mu_{\varepsilon,\pm}))u_{\pm}=0  \,\,\hbox{ for }\,\,x\in\R^d. $$

Next, we show that
\begin{equation}\label{u--U-u+}
u_-(x) \leq  U(x) \leq  u_+(x)  \,\,\hbox{ for all }\,\, x\cdot e \geq \xi_+.
\end{equation}
We only give the proof of $\underline{u}(x) \leq  U(x)$ for  $x\cdot e \geq \xi_+$, as we will sketch below that the other inequality can be proved similarly.   To do so, we call
$$\gamma^*:=\inf\{\gamma\geq 0:\, U(x)-u_-(x)+\gamma \psi^-(x) \geq 0  \hbox{ for }   x \cdot e \geq \xi_+ \},$$
where $\psi^-(x)$ is the unique principal eigenfunction of \eqref{eigenvalue0} satisfying the normalization condition \eqref{normalization}.
Since $U(x)$, $u_-(x)$ are bounded in $x\cdot e\geq \xi_+$ and since $\psi^-(x)$ is positive, it is clear that $\gamma^*\geq 0$ is well defined. It is also easily seen that, to obtain the first inequality of \eqref{u--U-u+}, one only needs to show that $\gamma^*= 0$.

Assume by contradiction that $\gamma^*>0$. Define
$$z(x):=U(x)-u_-(x)+\gamma^*\psi^-(x)\,\,\hbox{ for all }\,\, x\in\R^d. $$
Then, $z\geq0$ for all $x\cdot e\geq \xi_+$, and there exists a sequence $(x_n)_{n\in\N}\in \R^d$ such that $x_n\cdot e\geq \xi_+$
for each $n\in\N$, and that
\begin{equation}\label{zxn-infty}
z(x_n)\to 0 \,\,\hbox{ as} \,\, n\to \infty.
\end{equation}
Since $U(x)\to 0$ and $u_-(x)\to 0$  as $x\cdot e \to +\infty$ and since $\psi^-(x)$ is positive in $\R^d$,  the sequence $(x_n\cdot e)_{n\in\N}$  must be bounded. Moreover, it follows from \eqref{com-boundary} and the assumption $\gamma^*>0$ that $x_n\cdot e \geq \xi_++1$ for all $n\in\N$.

For each $n\in\N$, let us write $x_n=x_n'+x_n''$, where
$x_n'' \in{[0,L_1]\times\cdots\times[0,L_d]}$ and $x_n' \in L\Z^d$. It is clear that $(x_n'\cdot e)_{n\in\N}$ is bounded.
Then, up to extraction of some subsequence, one finds some $x_{\infty}\in{[0,L_1]\times\cdots\times[0,L_d]} $ and $\xi_{\infty}\in\R$ such that
$x_n'' \to x_{\infty}$ and $x_n'\cdot e\to \xi_{\infty}$ as $n\to\infty$. It is easily checked that $x_{\infty}\cdot e \geq \xi_++1-\xi_{\infty}$.
Furthermore,  for each $n\in\N$, we define
$$z_n(x):=z(x+x_n')\, \hbox{ for all } \,   x\in\R^d.  $$
Since the functions $x\mapsto A(x)$ and $x\mapsto f(x,u)$ are $L$-periodic, it follows from the first inequality of \eqref{derivatives-u} and \eqref{find-delta0}
that
\begin{equation*}
\left.
\begin{array}{ll}
 -\nabla \cdot (A(x)\nabla \!\!\!&\!\!\! U(x+x_n')) -(\partial_uf(x,0)+\lambda_1^-(\mu_{\varepsilon,-}))U(x+x_n') \vspace{5pt}\\
& =  f(x,U(x+x_n'))-(\partial_uf(x,0)+\lambda_1^-(\mu_{\varepsilon,-}))U(x+x_n') \geq 0
\end{array}\right.
\end{equation*}
for all $x\in\R^d $ such that $x\cdot e \geq \xi_+-x_n\cdot e$. Similarly, by the periodicity, we have
\begin{equation*}
-\nabla \cdot (A(x)\nabla u_-(x+x_n')) -(\partial_uf(x,0)+\lambda_1^-(\mu_{\varepsilon,-}))u_-(x+x_n')=0 \,\,\hbox{ in }\,\,\R^d.
\end{equation*}
Moreover, since $\psi^-(x)$ is the principal eigenfunction associated with $\lambda^-_1(0)$, it follows that
\begin{equation*}
-\nabla \cdot (A(x)\nabla (\gamma^*\psi^-(x)))-(\partial_uf(x,0)+\lambda_1^-(\mu_{\varepsilon,-}))(\gamma^*\psi^-(x))
= \gamma^* (\lambda^-_1(0)-\lambda_1^-(\mu_{\varepsilon,-}))\psi^-(x)
\end{equation*}
in $\R^d$. Combining the above, we obtain that, for each $n\in\N$,
\begin{equation}\label{eq-zn}
-\nabla \cdot (A(x)\nabla z_n(x)) -(\partial_uf(x,0)+\lambda_1^-(\mu_{\varepsilon,-}))z_n(x)\geq \gamma^* (\lambda^-_1(0)-\lambda_1^-(\mu_{\varepsilon,-}))\psi^-(x)    \end{equation}
for all $x\in\R^d$ such that $x\cdot e \geq \xi_+-x_n\cdot e$.
By standard elliptic estimates, there exists $z_{\infty}\in C^2(\R^d)$ such that,  up to extraction of some subsequence, $z_n\to z_{\infty}$ as $n\to \infty$ in $C^2_{loc}(\R^d)$. It is clear that $z_{\infty}(x)\geq 0$ for $x\cdot e \geq \xi_+-\xi_{\infty}$, and by \eqref{zxn-infty}, we have  $z_{\infty}(x_{\infty})=0$ (recall that $x_{\infty}\cdot e> \xi_+-\xi_{\infty}$). Moreover,  since $\lambda^-_1(0)>\lambda_1^-(\mu_{\varepsilon,-})$ and $\psi^-(x)$ is positive, passing to the limit as $n\to \infty$ in \eqref{eq-zn} yields
\begin{equation}\label{eq-zinfty}
-\nabla \cdot (A(x)\nabla z_{\infty}(x)) -(\partial_uf(x,0)+\lambda_1^-(\mu_{\varepsilon,-}))z_{\infty}(x)\geq \gamma^* (\lambda^-_1(0)-\lambda_1^-(\mu_{\varepsilon,-}))\psi^-(x) >0
\end{equation}
for $x\in\R^d$ such that $x\cdot e \geq \xi_+-\xi_{\infty}$.  It then follows from the strong maximum principle that $z_{\infty}(x)=0$ for all $x\cdot e \geq \xi_+-\xi_{\infty}$, which is a contradiction with  \eqref{eq-zinfty}.  Therefore, we obtain $\gamma^*=0$,  and the first part of \eqref{u--U-u+} is obtained.
The second part can be proved in a similar way, by defining this time
$$\gamma_*:=\inf\{\gamma\geq 0:\, u_+(x)-U(x)+\gamma \psi^-(x) \geq 0  \hbox{ for }   x \cdot e \geq \xi_+ \},$$
and using the second inequality \eqref{derivatives-u} and the fact that  $\lambda_1^-(\mu_{\varepsilon,+}) <  \lambda_1^-(0)$ to show $\gamma_*=0$.

Finally, taking $C_+=K_+$ and $C_-=K_- \min_{x\in\R^d}\psi_{\varepsilon,-}(x)$, we obtain \eqref{s-decay+}.
We omit the details for the estimates in the case $x\cdot e \leq 0$ whose proof is similar.
\end{proof}

\begin{rmk}\label{rem-decay}
By the concavity of the function $\mu\mapsto \lambda_1^-(\mu)$, the negative root (i.e., $\underline{\mu}$) of the equation $\lambda_1^-(\mu)=c^*\mu$ is decreasing in $c^*\in\R$. This together with the above two lemmas implies that the decay of a pulsating wave of \eqref{eq:main} going to $0$ becomes faster when the speed increases (the other limiting state of such a wave is not necessarily to be $1$).  Similarly, since the positive root (i.e., $\bar{\mu}$) of $\lambda_1^+(\mu)=c^*\mu$ is decreasing in $c^*$, for any pulsating wave of \eqref{eq:main} with $1$ being a limiting state,
the increase in the wave speed slows down the approach of the wave to $1$.
\end{rmk}

\subsection{Proof of Theorem \ref{Continuity of speeds}}\label{main-proof}

For any sequence $(e_n)_{n\in\N}\in \mathbb{S}^{d-1}$ such that $e_n\to e_0$ as $n\to \infty$ for some $e_0\in \mathbb{S}^{d-1}$, we need to show that $c^*(e_n)\to c^*(e_0)$ as $n\to+\infty$. In the arguments below, for convenience,  we denote $c^*_n:=c^*(e_n)$ and $c^*_0:=c^*(e_0)$, and
denote by $U_n$ (resp. $U_0$) a pulsating wave of \eqref{eq:main} connecting $0$ and $1$ in the direction $e_n$ (resp. $e_0$). We emphasize again that the non-stationary wave is unique up to shifts in time while the stationary wave may be not unique. Without loss of generality, we assume that
$$\int_0^1\int_{(0,L_1)\times \cdots\times(0,L_d)} f(x,u)dxdu \geq 0;$$
it then follows from the sign property of wave speeds stated in Theorem \ref{th:tw} (i) that
\begin{equation}\label{sign-speed}
c^*_n\geq 0\,\,\hbox{ for all }\,\, n\in\N\quad \hbox{and}\quad c^*_0\geq 0
\end{equation}
(in the case where $\int_0^1\int_{(0,L_1)\times \cdots\times(0,L_d)} f(x,u)dxdu \leq 0$, we have $c^*_n\leq 0$ for all $n\in\N$ and $c^*_0\leq 0$ which can be treated in a similar way).
Since the sequence $(c^*_n)_{n\in\N}$ is bounded by Theorem \ref{th:tw}, up to extraction of some subsequence, $c^*_n\to c_{\infty}$ as $n\to \infty$ for some $c_{\infty}\geq 0$. Therefore, to prove Theorem \ref{Continuity of speeds}, it suffices to show that
$$c^*_0=c_{\infty}.$$
For clarity, we will proceed with several steps. The arguments in the first step are similar to those used in the proof of \cite[Lemma 2.5]{dg}. For completeness and later use, we will include the details below.

We first give an elementary claim.

\smallskip

{\bf Claim 1.}
 There exists a small constant $\delta>0$ such that, for any stationary solution $u$ of \eqref{eq:main},
 if $0\leq u\leq \delta$, then $u\equiv 0$; if $1-\delta\leq u\leq 1$, then $u\equiv 1$.

We only prove the first assertion, since the second one is similar.  Recall that $\lambda^->0$ is the principal eigenvalue of \eqref{eigenvalue0} and $\psi^-(x)>0$ is the corresponding eigenfunction satisfying the normalization condition \eqref{normalization}.
Since the function $u \mapsto \partial_u f(x,u)$ is  continuous uniformly in $x \in\R^d$, there exists $\delta_0\in (0,1/2)$ such that
$$|\partial_uf(x,0)u-f(x,u)| \leq \frac{\lambda^-}{2}u  \,\, \hbox{ for }\,\, u\in [0,\delta_0],\,\,x\in\R^d.$$
Define $\bar{u}(t,x)=\delta_0e^{-\frac{\lambda^-}{2} t}\psi^-(x)$  for $t\geq 0$, $x\in\R^d$. It is clear that $0<\bar{u}(t,x)\leq \delta_0$ for $t\geq 0$, $x\in\R^d$. Moreover, direct computation gives that
$$\partial_t\bar{u}-\nabla \cdot (A(x)\nabla\bar{u})-f(x,\bar{u})=\left(\partial_uf(x,0)+\frac{\lambda^-}{2}\right)\bar{u}-f(x,\bar{u})\geq 0$$
for $t> 0$, $x\in\R^d$. Now, we choose $\delta:=\delta_0\min_{x\in\R^d} \psi^-(x)$. Then, if the stationary solution $u$ ranging between $0$ and $\delta$, there holds $0\leq u(x)\leq \bar{u}(0,x)$ for $x\in\R^d$. It thus follows from the comparison principle that $0\leq u(x)\leq \bar{u}(t,x)$ for $t\geq 0$, $x\in\R^d$. Passing to the limit as $t\to+\infty$, we immediately obtain $u\equiv 0$. This ends the proof of Claim 1.

\smallskip

{\bf Step 1.} If $c^*_0=0$, then $c_{\infty}=c^*_0=0$.

Assume by contradiction that $c_\infty \neq 0$. Then, by \eqref{sign-speed}, we have $c_\infty > 0$, and without loss of generality, we may assume that $c_n^*>0$ for all $n\in\N$. In this case,
the pulsating wave $U_n=U_n(t,x)$ is increasing in $t\in\R$ (see Theorem \ref{th:tw}), and hence, by continuity, there exists a unique $t_n\in\R$ such that
\begin{equation}\label{normal-UR-1-hd}
\max_{x\in[0,L_1]\times\cdots\times[0,L_d]}U_n(t_n,x)=\delta,
\end{equation}
where $\delta>0$ is the constant provided by Claim 1. Then, by standard parabolic estimates, there exists a function $U_{\infty}\in C^{1,2}(\R\times\R^d)$ such that, up to extraction of a subsequence, $U_n(t+t_n,x)\to U_{\infty}(t,x)$
in $C_{loc}^{1,2}(\R^{d+1})$ as $n\to +\infty$.  Clearly, $0\leq U_{\infty}(t,x)\leq 1$ is an entire solution of \eqref{eq:main}, and it is nondecreasing in $t\in\R$. Moreover, since $\max_{[0,L_1]\times\cdots\times[0,L_d]}U_{\infty}(0,\cdot)=\delta$, the strong maximum principle implies that $0<U_{\infty}<1$ in $\R\times \R^d$.
Furthermore, by the definition of non-stationary pulsating waves, for any $k\in \Z^d$, there holds
\begin{equation}\label{period-Un}
U_n\left(t+\frac{kL \cdot e_n}{ c^*_n}, \,x+kL\right) = U_n(t,x) \  \,\hbox{ for }\, t\in\R,\,x\in\R^d.
\end{equation}
Since $c_\infty >0$, taking the limit as $n\to +\infty$ gives that for any $k\in \Z^d$,
\begin{equation}\label{limit-pulsating-hd}
U_{\infty}\left(t+\frac{kL \cdot e_0}{c_\infty},\,x+kL\right)=U_{\infty}(t,x)  \ \,\hbox{ for } \,  t\in\R,\,x\in\R^d.
\end{equation}
By the monotonicity of $U_{\infty}$ in $t$ and standard parabolic estimates, there are two $L$-periodic steady states $u_{\pm}(x)$ of \eqref{eq:main} such that
\begin{equation*}
 U_{\infty}(t,x) \to  u_{\pm}(x)  \,\hbox{ as } \,  t\to\pm\infty\,  \hbox{ locally uniformly for }\, x\in\R^d.
 \end{equation*}
Since
$$0\leq \max_{[0,L_1]\times\cdots\times[0,L_d]}u_{-}(\cdot)  \leq \max_{[0,L_1]\times\cdots\times[0,L_d]}U_{\infty}(0,\cdot)=\delta\leq \max_{[0,L_1]\times\cdots\times[0,L_d]}u_{+}(\cdot)\leq 1,$$
it follows from Claim 1 that $u_{-}\equiv 0$.

Next, by the strong maximum principle, either $u_{+}\equiv 1$ or $0<u_{+}< 1$. The former case is impossible. Otherwise, $U_{\infty}(t,x)$ would be a non-stationary pulsating wave of \eqref{eq:main} connecting $0$ and $1$ in the direction $e_0$, while we have assumed that the wave in this direction is stationary (i.e., $c^*_0=0$), which is a contradiction with the uniqueness of wave speeds (see Theorem \ref{th:tw}).

It remains to consider the case where $0<u_{+}< 1$. Recall that $c^*_0=0$ and $U_0=U_0(x)$ is a stationary wave connecting $0$ and $1$ in the direction $e_0$.
We will find a contradiction by showing the following claim:

\smallskip

{\bf Claim 2.} There exists $t_0\leq  0$ such that
\begin{equation*}
U_0(x) \geq  U_{\infty}(t_0,x) \,\,\hbox{ in }\,\, \R^d.
\end{equation*}

For any $\mu\in\R$, let $\lambda_1^-(\mu)$ be the principal eigenvalue of \eqref{p-lambda-} with $e=e_0$, and let $\mu_1$, $\mu_2$ be, respectively, the negative roots of $\lambda_1^-(\mu)=0$ and $\lambda_1^-(\mu)=c_{\infty}\mu$. Since $c_{\infty}>0$, one sees from Remark \ref{rem-decay} that  $\mu_2<\mu_1<0$. Denote $\eta_0:=(\mu_1-\mu_2)/4>0$.
Then, on the one hand, since $U_0(x)$ is a stationary wave connecting $0$ nd $1$, it follows from Lemma \ref{exp-decay-hd-2} that there exists some constant $C_1>0$ such that
$$U_0(x)\geq C_1 e^{(\mu_1-\eta_0)x\cdot e_0}\,\,\hbox{ for all }\,\,x\cdot e_0 \geq 0. $$
On the other hand, notice that $U_{\infty}(t,x)$ is a non-stationary pulsating wave of \eqref{eq:main} connecting $0$ and $u_+$ (with $0$ being linearly stable). By Lemma \ref{exp-decay-hd-1}, one finds some constant $C_2>0$ such that
$$ U_{\infty}(t,x) \leq  C_2 e^{(\mu_2+\eta_0)(x\cdot e_0-c_{\infty}t)} \,\,\hbox{ for all }\,\, x\cdot e_0-c_{\infty}t\geq 0.$$
Since $\mu_2+\eta_0<\mu_1-\eta_0$, this in particular implies that $U_{\infty}(0,x)$ decays faster than $U_0(x)$ as $x\cdot e_0\to +\infty$. Therefore, there exists some $M_1>0$ such that
\begin{equation}\label{U-Uinfty}
U_0(x) \geq   U_{\infty}(0,x)  \,\, \hbox{ for all }\,\, x\cdot e_0\geq M_1.
\end{equation}

Now, suppose by contradiction that Claim 2 is not true. Then, by \eqref{U-Uinfty} and the fact that $U_{\infty}(t,x)$ is nondecreasing in $t\in\R$, one finds sequences $(t_n)_{n\in\N}\subset (-\infty,0]$ and $(x_n)_{n\in\N}\subset\R^d$ such that $t_n\to -\infty$ as $n\to +\infty$, and that for each $n\in\N$, $x_n \cdot e_0 < M_1$ and
$U_0(x_n) <  U_{\infty}(t_n,x_n)$.
However, this is impossible. Indeed, if the sequence $(x_n \cdot e_0)_{n\in\N}$ is bounded, then $\liminf_{n\to +\infty}U_0(x_n)>0$, while $U_{\infty}(t_n,x_n)\to 0$ as $n \to +\infty$ (since $u_-\equiv 0$). On the other hand, if $(x_n \cdot e_0)_{n\in\N}$ is unbounded, then
$\limsup_{n\to +\infty}U_0(x_n) = 1$, while
$0< U_{\infty}(t_n,x_n)\leq u_+(x_n)$ for all $n\in\N$ (recall that $u_+$ is a periodic function strictly between $0$ and $1$), whence $\limsup_{n\to+\infty}U_{\infty}(t_n,x_n)<1$.
This completes the proof of Claim 2.

It then follows from the comparison principle that
$$U_0(x) \geq U_{\infty}(t,x)  \,\,\hbox{ for all }\, \, t\geq t_0,\,x\in\R^d.$$
Passing to the limit as $t \to +\infty$, we get that $U_0\geq u_+$ in $\R^d$, which contradicts the fact that $U_0(x)\to 0$ as $x\cdot e_0\to +\infty$. This means that our assumption $c_{\infty}>0$ at the beginning was false. Therefore, $c_{\infty}=0$ and Step 1 is complete.

In the remaining steps, we prove $c^*_0=c_{\infty}$ in the case where $c^*_0>0$. We first show that:

\smallskip

{\bf Step 2}: If $c^*_0>0$, then $c^*_n>0$ for all large $n\in\N$.

Assume by contradiction that $c_n=0$ for all large $n\in\N$. Then, by the definition of stationary pulsating waves, for each large $n\in\N$, $U_n(x)\to 1$ as $x\cdot e_n\to -\infty$ and $U_n(x)\to 0$ as $x\cdot e_n\to +\infty$. This implies that there exists some $x_n\in \R^d$ such that
$U_n(x_n)=1-\delta$, and that
$$ 1-\delta \leq U_n(x) <1 \,\,\hbox{ for all }  x\in\R^d\,\,\hbox{ such that }\,\, x\cdot e_n \leq x_n\cdot e_n-1,  $$
where $\delta>0$ is given by Claim 1.
Write $x_n=x_n'+x_n''$ with $x_n'\in [0,L_1]\times\cdots\times[0,L_d]$ and $x_n''\in L\Z^d$, and define
$$V_n(x):=U_n(x+x_n'')\,\,\hbox{ for }\,\,x\in\R^d. $$
By the periodicity, $V_n(x)$ remains a stationary pulsating wave of \eqref{eq:main} in the direction $e_n$. It is also easily seen that $V_n(x_n')= 1-\delta$ and
$$ 1-\delta \leq V_n(x) <1 \,\,\hbox{ for all }  x\in\R^d\,\,\hbox{ such that }\,\, x\cdot e_n \leq -\kappa-1, $$
where $\kappa>0$ is a constant given by
\begin{equation}\label{de-kappa}
\kappa=\sqrt{\sum_{i=1}^d L^2_i}.
\end{equation}
Up to extraction of some subsequence, we may assume that $x_n'\to x_{\infty}$ as $n\to+\infty$, and by standard elliptic estimates, there is a $C^2(\R^d)$ function $V_{\infty}$ such that
$V_n\to V_{\infty}$ in $C^2_{loc}(\R^d)$ as $n\to+\infty$. Clearly, $V_{\infty}(x)$ is an entire solution of \eqref{eq:main} satisfying
\begin{equation}\label{Vinfty-delta}
V_{\infty}(x_{\infty})=1-\delta,\quad \hbox{and}\quad  1-\delta\leq  V_{\infty}(x)\leq 1 \,\,\hbox{ for all }\,\, x\cdot e_0\leq -\kappa-1.
\end{equation}
Furthermore, by the strong maximum principle, we have $0<V_{\infty}<1$ in $\R^d$. Next, we show the following claim:

\smallskip

{\bf Claim 3.} $V_{\infty}(x)\to 1$ as $x\cdot e_0\to -\infty$.

Let $(y_m)_{m\in\N} \in\R^d$ be an arbitrary sequence such that $y_m\cdot e_0\to -\infty$ as $m\to+\infty$. It suffices to show that $V_{\infty}(y_m)\to 1$ as $m\to+\infty$. For each $m\in\N$, write $y_m=y'_m+y''_m$ with  $y'_m\in [0,L_1]\times\cdots\times[0,L_d]$ and $y''_m\in L\Z^d$, and define $W_m(x):=V_{\infty}(x+y_m'')$ for $x\in\R^d$. It is clear that $y_m''\cdot e_0 \to -\infty$ as $m\to+\infty$, and for each $m\in\N$, $W_m(x)$ is a stationary solution of \eqref{eq:main}. Then, by standard elliptic estimates, there exists a subsequence $(W_{m_j})_{j\in\N}$ converging in $C^2_{loc}(\R^2)$ to some function $W_{\infty}\in C^2(\R)$ which remains a stationary solution of \eqref{eq:main}. Due to \eqref{Vinfty-delta}, one checks that $1-\delta\leq W_{\infty}\leq 1$ in $\R^d$. It then follows from Claim 1 that $W_{\infty}\equiv 1$. By the uniqueness of the limit, the whole sequence $W_m$ converges to $1$ in $C^2_{loc}(\R^2)$, whence by the boundedness of the sequence $(y_m')_{m\in\N}$, we obtained that $V_{\infty}(y_m)\to 1$ as $m\to+\infty$.

Now, recall that $c^*_0>0$ and $U_0(t,x)$ is a non-stationary wave connecting $0$ and $1$ in the direction $e_0$. By the definition of pulsating waves, Claim 2 and the fact that $0<V_{\infty}<1$, for any small $\epsilon_0>0$,
one finds some $k_0\in\Z^d$ (with $k_0L\cdot e_0$ sufficiently large) such that
$$V_\infty(x)\geq U_0(0,x+k_0L)-\varepsilon_0 [\rho(x\cdot e_0)\psi^+(x)+(1-\rho(x\cdot e_0)-ct))\psi^-(x)] \,\, \hbox{ for all }\,\,  x\in \R^d,$$
where $\rho\in C^2(\R,[0,1])$ is a function satisfying \eqref{rho}, and $\psi^+$, $\psi^-$ are, respectively, the principal eigenfunctions of
\eqref{eigenvalue1} and \eqref{eigenvalue0} with the normalization \eqref{normalization}. It then follows from Lemma \ref {sub-super-nonzero} and the comparison principle that
 \begin{equation}\label{Vinfty-U0}
 V_\infty(x) \geq  U_0(t-\varepsilon_0 K_0(1-e^{-\mu_0 t}), x+k_0L)-\varepsilon_0 e^{-\mu_0 t} \,\,\hbox{ for }\,\, t\geq 0,\, x\in\R^d,
 \end{equation}
for some constants $\mu_0>0$ and $K_0>0$ (notice that $c_0^*>0$). Passing to the limit as $t\to+\infty$ in the above inequality, we get that $V_{\infty}\geq 1$ in $\R^d$, which is impossible.
Therefore, $c_n^*$ must be positive for all large $n$.

As a consequence, in what follows, we may assume without loss of generality that $c_n>0$ for all $n\in\N$.

\smallskip

{\bf Step 3.} If $c^*_0>0$, then $c_\infty >0$.

Notice that each $U_n$ is non-stationary and increasing in the first variable. By continuity, we find a unique $s_n\in\R$ such that
\begin{equation}\label{normal-UR-2-hd}
\min_{x\in[0,L_1]\times\cdots\times[0,L_d]}U_n(s_n,x)=1-\delta.
\end{equation}
By standard parabolic estimates, up to extraction of a subsequence, the functions $(t,x)\mapsto U_n(t+s_n,x) $ converge in $C_{loc}^{1,2}(\R\times\R^d)$ to an entire solution $0< \tilde{U}_{\infty}(t,x)< 1$ of~\eqref{eq:main}, such that $\min_{x\in [0,L_1]\times\cdots\times[0,L_d]}\tilde{U}_{\infty}(0,x)=1-\delta$ and $\tilde{U}_{\infty}(t,x)$ is nondecreasing in $t\in\R$. It also follows that there exists a stationary solution $0\leq \tilde{u}_-\leq 1$ of \eqref{eq:main} such that
$ \tilde{U}_{\infty}(t,x)\to \tilde{u}_-(x)$ as  $t\to -\infty$ in $C_{loc}^{2}(\R^d)$.
Since
$$ \min\limits_{[0,L_1]\times\cdots\times[0,L_d]}\tilde{u}_-(\cdot)\leq \min\limits_{[0,L_1]\times\cdots\times[0,L_d]}\tilde{U}_{\infty}(0,\cdot)=1-\delta,$$
by the strong maximum principle, we have $\tilde{u}_-<1$ in $\R^d$.

Now, by supposing to the contrary that $c_\infty = 0$, we show
\begin{equation}\label{tildeu-infty}
\tilde{u}_-(x)\to 1 \,\,\hbox{ as }\,\, x\cdot e_0\to -\infty.
\end{equation}
Let $t\in\R$ and $x\in\R^d$ satisfying $x\cdot e_0 < -\kappa-1$ be arbitrarily fixed, where $\kappa>0$ is given by \eqref{de-kappa}.
It is clear that there exists $k_*\in \Z^d$ such that $x-k_*L\in [0,L_1]\times\cdots\times[0,L_d]$. Since $e_n\to e_0$ as $n\to\infty$, one has $x\cdot e_n \leq -\kappa-1$ for all large $n$, whence $k_*L\cdot e_n\leq -1$. Since $c_n^*>0$ for all $n\in\N$ and since $c_n^*\to c_{\infty}=0$ as $n\to +\infty$, this implies that
$t-k_*L\cdot e_n/c_n^*\geq 0$ for all $n$ large enough.
Then, by \eqref{period-Un} and the monotonicity of $U_n$ with respect to $t$, it follows that
\begin{equation*}
\left.\begin{array}{ll}
U_n(t+s_n,x)\!\!\!&=\displaystyle U_n(t+s_n-\frac{k_*L\cdot e_n}{c_n^*},x-k_*L) \vspace{5pt}\\
&\geq \min \limits_{[0,L_1]\times\cdots\times[0,L_d]}U_n(s_n,\cdot)=1-\delta.
\end{array}\right.
\end{equation*}
Passing to the limit as $n\to +\infty$ followed by letting $t\to -\infty$, we obtain that
$$1-\delta \leq \tilde{u}_-(x)<1\,\, \hbox{ for all } \,\, x\cdot e_0 <-\kappa-1.$$
Then, \eqref{tildeu-infty} follows directly from the proof of Claim 3.

As a consequence, the same reasoning in showing \eqref{Vinfty-U0} implies that there exists some $\tilde{k}_0\in \Z^d$ such that
 $$\tilde{u}_-(x) \geq  U_0(t-\varepsilon_0 K_0(1-e^{-\mu_0 t}), x+\tilde{k}_0L)-\varepsilon_0 e^{-\mu_0 t} \,\,\hbox{ for }\,\, t\geq 0,\, x\in\R^d.$$
We now reach a contradiction by letting $t\to +\infty$. Therefore, we have $c_{\infty}>0$.  As a consequence,
$\tilde{U}_{\infty}$  satisfies \eqref{limit-pulsating-hd}, and hence, $0\leq \tilde{u}_-<1$ is an $L$-periodic steady state of \eqref{eq:main}.

\smallskip

{\bf Step 4.} If $c^*_0>0$, then $c^*_0=c_{\infty}$.

Assume by contradiction that $c_\infty \neq c^*_0$. Then by Step 3,
$$\hbox{ either }\,\, c_\infty> c^*_0>0 \,\,\hbox{ or } \,\, 0<c_\infty< c^*_0.$$
Let us first derive a contradiction in the former case.
Since each $U_n$ is non-stationary, one finds a sequence $(t_n)_{n\in\N}\subset \R$ satisfying \eqref{normal-UR-1-hd}.
Proceeding similarly as in Step 1, we obtain a pulsating wave $U_{\infty}(t,x)$ of \eqref{eq:main} with speed $c_{\infty}$ in the direction $e_0$, and it either connects $0$ and $1$ or it connects $0$ and $u_+$ ($u_+$ ranges strictly between $0$ and $1$). The former is impossible, due to the uniqueness of wave speeds for bistable pulsating waves (notice that $U_0$ is a pulsating wave with speed $c^*_0$ connecting $0$ and $1$ in the same direction). In the latter case, since the decay of a pulsating wave going to 0 becomes faster when the speed increases (see Remark \ref{rem-decay}), it follows from similar arguments to those used in the proof of Claim 2 that
$ U_{\infty}(\tau_0,x)\leq  U(0,x)$ in $\R^d$ for some $\tau_0\in\R$. Then, the comparison principle implies that
$$U_{\infty}(t+\tau_0,x)\leq  U_0(t,x) \ \,\hbox{ for }\, t\geq 0,\, x\in\R^d .$$
This contradicts our assumption $ c_\infty >c^*_0$.

On the other hand, in the case where $0<c_\infty< c^*_0$, we take the normalization condition \eqref{normal-UR-2-hd} and let $\tilde{U}_{\infty}(t,x)$ be the function obtained in Step 3. Similarly, $\tilde{U}_{\infty}(t,x)$ is a pulsating wave either connecting $0$ and $1$ or connecting $\tilde{u}_-$ ( $\tilde{u}_-$ is an $L$-periodic steady state of \eqref{eq:main} strictly between $0$ and $1$) and $1$. As argued above, the former contradicts the uniqueness of bistable wave speeds; the latter also leads to a contradiction, as one can find some $\tilde{\tau}_0\in\R$ such that
$$\tilde{U}_{\infty}(t+\tilde{\tau}_0,x)\geq  U_0(t,x) \ \,\hbox{ for }\, t\geq 0,\, x\in\R^d .$$

Combining the above, we reach that $c^*_0=c_{\infty}$. The proof of Theorem \ref{Continuity of speeds} is thus complete.

\section{Existence of pulsating waves in rapidly oscillating media}

\subsection{Proof of Proposition \ref{condi}}

We show Proposition \ref{condi} by using the notion of propagating terrace which consists of a finite sequence of stacked waves. The notion was first introduced by \cite{dgm} in the spatially periodic case. For our homogenous equation \eqref{onehom}, the definition of propagating terrace is presented as follows.

\begin{defi}\label{terrace}
A propagating terrace of \eqref{onehom} connecting $0$ and $1$ is a pair of finite sequences $(p_k)_{0\leq k\leq N}$ and $(V_k)_{1\leq k\leq N}$ such that 
\begin{itemize}
\item  Each $p_k$ is a zero of $\bar{f}$ satisfying  $1=p_0>p_1>\cdots>p_N=0$; 
\item For each $1\leq k \leq N$, $V_k(\cdot)$ is a traveling wave of \eqref{onehom} connecting $p_k$ and $p_{k-1}$ $($$V_k(-\infty)=p_{k-1}$ and $V_k(+\infty)=p_k$$)$ with speed $c_k\in\R$; 
\item The sequence of speeds $(c_k)_{1\leq k\leq N}$ satisfies $c_1\leq c_2 \leq \cdots\leq c_N$. 
\end{itemize}
\end{defi} 

To prove Proposition \ref{condi}, we first give the existence of propagating terrace connecting $0$ and $1$, and then show that the terrace consists of a single traveling wave. The former part is ensured by the following lemma, which follows directly from \cite[Theorem 1]{po2} together with its followed discussion. Recall that $F: [0,1]\to \R$ is the primitive function of $\bar{f}$ defined in \eqref{de-F}.  

\begin{lem}[\cite{po2}]\label{exist-terace}
If the point $u = 1$ is a unique maximizer of the function $F$ in $[0,1]$ and it is an isolated zero of $\bar{f}$ in $[0,1]$, then equation \eqref{onehom} admits a propagating terrace connecting $0$ and $1$ with positive traveling wave speeds. Similarly, 
if $u=0$ is a unique maximizer of $F$ in $[0,1]$ and an isolated zero of $\bar{f}$ in $[0,1]$, then there exists a propagating terrace 
connecting $0$ and $1$ with negative speeds.  
\end{lem}

We remark that the terraces obtained above are unique up to translations and globally asymptotically stable. As these properties are not needed in showing our results, we do not state them here.

\begin{proof}[Proof of Proposition \ref{condi}]
We first assume condition (a) and  show the existence of traveling wave connecting $0$ and $1$ with negative speed. 
Clearly, under condition (a), by the continuity of the function $F:[0,1]\to \R$, two possibilities may happen: 
\begin{itemize}
\item $F(u)$ is decreasing in $[0,1]$; 
\item there exists some $u^*\in (0,1)$ such that 
\begin{equation}\label{find-u-s}
F(u^*)=F(1)<0,\,\,\,  \bar{f}< 0\, \hbox{ for } \, u\in (0,u^*]\,\,\, \hbox{and}\,\,\, F(u)\leq  F(u^*)\,\hbox{ for } \, u\in (u^*,1).
\end{equation}
\end{itemize}

In the former, $u=0$ and $u=1$ are the only two zeros of $\bar{f}$ in $[0,1]$, and $\bar{f}$ is of the monostable type ($u=0$ is stable while $u=1$ is unstable). It is then well known \cite{aw} that there exists $c_*<0$ such that \eqref{onehom} admits a traveling wave connecting $0$ and $1$ with speed $c$ if and only if $c\leq c_*$. Then, nothing is left to prove in this case. 

In the latter,  $u=0$ is a unique maximizer of $F$ in $[0,1]$ and an isolated zero of $\bar{f}$ in $[0,1]$.  Then by Lemma \ref{exist-terace}, there exists a propagating terrace $((p_k)_{0\leq k\leq N},(V_k)_{1\leq k\leq N})$ connecting 0 and 1 with speeds  $c_1\leq c_2 \leq \cdots \leq c_N < 0$. It remains to show that $N=1$. Assume by contradiction that $N> 1$. Then $V_1(-\infty)=1$ and $V_1(+\infty)=p_1$ with $p_1 \in (u^*,1)$, since $\bar{f}$ has no zero in $(0,u^*)$. By the last inequality in \eqref{find-u-s}, we have $F(1)\geq F(p_1)$. Moreover, since $V_1$ is a traveling wave of \eqref{onehom}, 
it satisfies $A_0V_1''+c_1V_1'+\bar{f}(V_1)=0$ in $\R$. Multiplying this equation by $V_1'$ and integrating over $\R$ gives 
$$sign(c_1)=sign(F(1)-F(p_1))\geq 0, $$
which is a contradiction with the fact $c_1<0$. Therefore, $N=1$, and $V_1$ is a traveling wave of \eqref{onehom} connecting $0$ and $1$ with speed $c_1<0$. 

The situation with positive wave speed can be handled analogously. Indeed, condition (b) implies that either $F(u)$ is increasing in $[0,1]$ or
$u=1$ is a unique maximizer of $F$ in $[0,1]$ and an isolated zero of $\bar{f}$ in $[0,1]$. 
In the former,  \eqref{onehom} admits a family of traveling waves connecting $0$ and $1$ with positive speeds. In the latter,  Lemma \ref{exist-terace} gives the existence of propagating terrace $((\tilde{p}_k)_{0\leq k\leq N},(\tilde{V}_k)_{1\leq k\leq N})$ with speeds $0<\tilde{c}_1\leq \tilde{c}_2 \leq \cdots \leq \tilde{c}_N$. Proceeding similarly as above, by working this time with the wave $\tilde{V}_N$, one can conclude that $N=1$. This ends the proof of Proposition \ref{condi}. 
\end{proof}

\subsection{Proof of Theorem \ref{thhomo}}

In this subsection, we fix a direction $e\in \mathbb{S}^{d-1}$, assume that \eqref{onehom} admits a traveling wave $\phi_0$ with speed $c_0\neq 0$, 
and show  the existence of non-stationary pulsating wave of type \eqref{exist-small} for \eqref{eql} when $l$ is small.  Without loss of generality, we assume that $c_0>0$. 

Upon substitution, a non-stationary pulsating wave $(\phi_l, c^*_l)$ of \eqref{eql} satisfies 
\begin{equation}\label{pulsolu}\left\{\baa{l}
\tilde{\partial}_{l}\cdot(A(y)\tilde{\partial}_{l}\phi_{l})+c^*_{l}\partial_{\xi}\phi_{l}+f(y,\phi_{l})=0 \hbox{ for all }(\xi,y)\in\R\times\R^d,\vspace{3pt}\\
\phi_{l}(\xi,y)\hbox{ is $L$-periodic in $y\in\R^d$},\vspace{3pt}\eaa\right.
\end{equation}
with the limiting conditions  
\begin{equation}\label{con-limit}
\phi_{l}(-\infty,y)=1,\quad \phi_{l}(+\infty,y)=0\,\,\hbox{ uniformly in }\,\, y\in\R^d, 
\end{equation}
where
$$\tilde{\partial}_{l} = e\partial_{\xi}+\frac{1}{l}\nabla_{y}.$$
Below, we use the implicit function theorem to derive the existence of $(\phi_l, c^*_l)$ with $c^*_l\neq 0$  when $l$ is small. Since the proof follows the same lines as those used in the one-dimensional case \cite[Theorem 1.2]{dhz1}, we only give its outline, and provide the details when considerable modifications are needed. 

We set some notations. Let $L^{2}(\R\times\T^d)$ and $H^{1}(\R\times \T^d)$ be two Banach spaces given by 
$$\baa{rcl}
L^{2}(\R\times\T^d)=\{v\in L^{2}_{loc}(\R\times \R^d):\ v\in L^{2}(\R\times (0,L_1) \times \cdots \times (0,L_d))  \vspace{3pt}\\
\hbox{ and } v(\xi,y) \hbox{ is $L$-periodic in $y$ almost everywhere in } \R^{d+1}  \},
\eaa$$
$$\baa{rcl}
H^{1}(\R\times\T^d)=\{v\in H^{1}_{loc}(\R\times \R^d):\ v\in H^{1}(\R\times (0,L_1) \times \cdots \times (0,L_d))  \vspace{3pt}\\
\hbox{ and } v(\xi,y) \hbox{ is $L$-periodic in $y$ almost everywhere in } \R^{d+1}  \},
\eaa$$
endowed with the norms  $\Vert v\Vert_{L^{2}(\R\times \T^d)} =\Vert v\Vert_{L^{2}(\R\times (0,L_1)\times \cdots (0,L_d))}$ and $$\Vert v\Vert_{H^{1}(\R\times \T^d)}=\left(\Vert v\Vert_{L^{2}(\R\times\T^d)}^{2}+\Vert \partial_{\xi}v\Vert_{L^{2}(\R\times\T^d)}^{2} +\sum_{i=1}^{d}\Vert \partial_{y_{i}}v\Vert_{L^{2}(\R\times \T^d)}^{2}\right)^{1/2}.$$
Fix a real $\beta>0$ and for any $c>0$ and $l\in \R^*:=\R\setminus\{0\}$, define 
$$M_{c,l}(v)=\tilde{\partial}_{l}\cdot(A(y)\tilde{\partial}_{l}v)+c\partial_{\xi}v-\beta v$$
for $v\in \mathcal{D}_{l}:=\{v\in H^{1}(\R\times \T^d):\, \tilde{\partial}_{l}\cdot(A(y)\tilde{\partial}_{l}v)\in L^{2}(\R\times \T^d)\}$, and
$$M_{c,0}(v)=A_{0}v''+cv'-\beta v \,\,\hbox{ for } \,\,v\in H^2(\mathbb{R}).$$
 By the proof of \cite[Lemma 3.1]{dhz1},
the operators $M_{c,0}: H^2(\R)\, \rightarrow \, L^2(\R)$ and $M_{c,l}: \mathcal{D}_l\, \rightarrow \, L^2(\R\times \T^d)$ for $l\neq0$ are invertible for every $c>0$. Furthermore, for every $r_1> 0$ and $r_2>0$, there is a constant $C=C(r_1,r_2,\beta,A)$ such that for all $c\geq r_1$, $|l|\leq r_2$, $g\in L^2(\R\times\T^d)$ and $\varpi\in L^2(\R)$,
\begin{equation}\label{inverse-1}
\left\{
\begin{array}{l}
\|M_{c,l}^{-1}(g)\|_{H^1(\R\times\T^d)} \leq C\|g\|_{L^2(\R\times\T^d)}, \vspace{5pt}\\
\|\nabla_{y} M_{c,l}^{-1}(g)\|_{L^2(\R\times\T^d)} \leq C|l|\|g\|_{L^2(\R\times\T^d)},
\end{array}\right. \hbox {if }\,\,l\neq 0,
\end{equation}
and 
\begin{equation}\label{inverse-2}
\|M_{c,0}^{-1}(\varpi)\|_{H^1(\R)} \leq C\|\varpi\|_{L^2(\R)}.
\end{equation}

Next, we check that the operator $M_{c_n,l_n}^{-1}$ converges to $M_{c,l}^{-1}$  as $n\to\infty$ when $c_n\to c\in (0,+\infty)$ and $l_n\to l\in\R$. In the case where $l\neq 0$, the verification is identical to that of \cite[Lemma 3.3]{dhz1}; therefore, we omit the details.  On the other hand, if $l= 0$, the proof is inspired by that of \cite[Lemma 3.2]{dhz1} but considerable modifications are needed. Moreover, from this case, one can see the role that the function $\chi$ (determined by \eqref{de-chi}) plays in the homogenization process for the multi-dimensional equation \eqref{eql}. 

\begin{lem}\label{continuem}
Fix $\beta>0$ and $c>0$. For any sequences $(g_n)_{n\in\N}$ in $L^2(\R\times\T^d)$, $(\varpi_n)_{n\in\N}$ in $L^2(\R)$, $(c_n)_{n\in\N}$ in $(0,+\infty)$ and~$(l_n)_{n\in\N}$ in $\R^*$ satisfies $\|g_n-g\|_{L^2(\R\times\T^d)}\to0$, $\|\varpi_n-\bar{g}\|_{L^2(\R)}\to0$, $c_n\to c$ and $l_n\to0$ as~$n\to+\infty$, then when $n\to+\infty$, 
\be\label{convinverse1}
M_{c_n,l_n}^{-1}(g_n) \to M_{c,0}^{-1}(\bar{g})\,\, \hbox{ in } \,\,H^1(\R\times\T^d),\vspace{3pt}
\ee
and
\be\label{convinverse2}
M_{c_n,0}^{-1}(\varpi_n) \to  M_{c,0}^{-1}(\bar{g})\,\, \hbox{ in } \,\, H^1(\R),
\ee
where $\bar{g}\in L^2(\R)$ is defined as \eqref{average}.  Furthermore, the limits \eqref{convinverse1} and  \eqref{convinverse2} are uniform in the ball $B_M=\big\{g\in H^1(\R\times\T^d)\,|\,\|g\|_{H^1(\R\times\T^d)}\le M\big\}$ for every $M>0$.
\end{lem}

\begin{proof}
We only consider the convergence \eqref{convinverse1}, since \eqref{convinverse2} is similar and simpler. 
Below, we prove \eqref{convinverse1} in the special case where $g_n=g$ and $c_n=c$ are independent of $n$. Once this is achieved, the arguments for the general case are the same as those used in Step 2 of the proof of \cite[Lemma 3.2]{dhz1}.

For each $n\in\N$, let $v_n=M_{c,l_n}^{-1}(g)\in H^1(\R\times\T^d)$. Since the sequence $(l_n)_{n\in\N}$ is bounded, 
it follows from the first line of \eqref{inverse-1} that $(v_n)_{n\in\N}$ is bounded in $H^1(\R\times\T^d)$. Then there exists a function $v_0 \in H^1(\R\times\T^d)$ such that, up to extraction of some subsequence, 
 $v_{n}\to v_0$ strongly in $L^2_{loc}(\R\times \T^d)$ and $v_{n}\rightharpoonup v_0$ weakly in $H^1(\R\times \T^d)$. 
 Thanks to the second line of \eqref{inverse-1},  we have $\nabla_yv_0=0$. Thus, the function $v_0$ can be viewed as an $H^1(\R)$ function and we can set $v_0'=\partial_{\xi}v_0$. For any $\phi\in H^2(\R)$, choosing $\psi(\xi,y)=\phi(\xi)+l_n\chi(y)\phi'(\xi)\in H^1(\R\times\T^d)$ as a test function in~$M_{c,l_n}(v_n)=g$ gives 
$$\int_{\R}-\kern-10pt\int_{\T^d}((\nabla_y\chi+e)\phi'+el_n\chi\phi'')  A\tilde{\partial}_{l_n}v_n+\int_{\R}-\kern-10pt\int_{\T^d}
(-c\partial_{\xi}v_n+\beta v_n)(\phi+l_n\chi\phi')=-\int_{\R}-\kern-10pt\int_{\T^d} g(\phi+l_n\chi\phi').$$
Since $l_n\to 0$, it follows from the weak convergence of $v_n$ to $v_0$ in  $H^1(\R\times \T^d)$ that 
\begin{equation*}
\begin{split}
\int_{\R}-\kern-10pt\int_{\T^d} ((\nabla_{y}\chi+e)\phi'+&el_{n}\chi\phi'') A\tilde{\partial}_{l_{n}}v_{n}\vspace{3pt}\\
& = \int_{\R}-\kern-10pt\int_{\T^d}-v_{n}\tilde{\partial}_{l_{n}}\cdot (A(\nabla_{y}\chi+e)\phi')+(el_{n}\chi\phi'') A\tilde{\partial}_{l_{n}}v_{n}\vspace{3pt}\\
& = \int_{\R}-\kern-10pt\int_{\T^d}eA(\nabla_{y}\chi+e)\phi'\partial_{\xi}v_{n}+(e\chi\phi'') A(el_{n}\partial_{\xi}v_{n}+\nabla_{y}v_{n})\vspace{3pt}\\
&\xrightarrow{n\to+\infty} \int_{\R}-\kern-10pt\int_{\T^d} eA(\nabla_{y}\chi+e)v_0'\phi'=A_0\int_{\R}v_0'\phi'.
\end{split}
\end{equation*}
Moreover, it is clear that 
$$\int_{\R}-\kern-10pt\int_{\T^d} (-c\partial_{\xi}v_n+\beta v_n)(\phi+l_n\chi\phi')\to\int_{\R}(-cv'_0+\beta v_0)\phi,$$ 
and 
$$\int_{\R}-\kern-10pt\int_{\T^d} g(\phi+l_n\chi\phi')\to\int_{\R} \bar{g}\phi$$ 
as $n\to+\infty$. Thus, $\int_{\R}A_0v_0'\phi'-cv_{0}'\phi+\beta v_0\phi=-\int_{\R} \bar{g}\phi$ for all $\phi\in H^2(\R)$ and then for all $\phi\in H^1(\R)$ by density. This means that $v_0$ is a weak solution of $A_0v''+cv'-\beta v=\bar{g}$ in $\R$. By the regularity theory, $v_0\in H^2(\R)\cap C^2(\R)$, and $M_{c,0}(v_0)=\bar{g}$.

In the arguments below, we show that $v_n$ converges to $v_0$ in $H^1(\R\times\T^d)$ as $n\to+\infty$, and this convergence is uniform with respect to $g$ in the ball $ B_M=\big\{g\in H^1(\R\times\T^d)\,:\,\|g\|_{H^1(\R\times\T^d)}\le M\big\}$.  For each $n\in\N$, let 
\begin{equation*}
\zeta_n(\xi,y)=v_n(\xi,y)-v_0(\xi)-l_n\chi(y)v_0'(\xi).
\end{equation*}
Since the sequence $(v_n)_{n\in\N}$ is bounded in $H^1(\R\times\T^d)$, $v_0\in H^2(\R)$ and $\chi(y)\in C^2(\R^d)$, the sequence $(\zeta_n)_{n\in\N}$ is bounded in $H^1(\R\times\T^d)$. Since $l_n\to 0$ as $n\to\infty$, to complete the proof, it suffices to show that $\zeta_n$ converges to $0$ in $H^1(\R\times\T^d)$ as $n\to\infty$ uniformly for $g\in B_M$.
First fo all, by the second line of \eqref{inverse-1}, $\nabla_y v_n$ converges to $0$ uniformly with respect to $g\in B_M$, and hence, so does $\nabla_y \zeta_n$.

Next, choosing $\zeta_n$ as a test function in $M_{c,l_n}(v_n)=g$ gives 
$$\int_{\R}-\kern-10pt\int_{\T^d}(\tilde{\partial}_{l_n}\zeta_n)  A\tilde{\partial}_{l_n}v_n-c\zeta_n\partial_{\xi}v_n+\beta v_n\zeta_n=-\int_{\R}-\kern-10pt\int_{\T^d} g\,\zeta_n.$$
It is straightforward to compute that
\begin{equation*}
\begin{split}
\int_{\R}-\kern-10pt\int_{\T^d} &(\tilde{\partial}_{l_n}\zeta_n) A\tilde{\partial}_{l_n}v_n \\
& =\int_{\R}-\kern-10pt\int_{\T^d}(\tilde{\partial}_{l_n}\zeta_n) A\tilde{\partial}_{l_n}\zeta_n+(\tilde{\partial}_{l_n}\zeta_n) Aev'_0+(\tilde{\partial}_{l_n}\zeta_n) A\nabla_y\chi v'_0+l_n(\tilde{\partial}_{l_n}\zeta_n) Ae\chi v''_0\vspace{3pt}\\
&=\int_{\R}-\kern-10pt\int_{\T^d}(\tilde{\partial}_{l_n}\zeta_n) A\tilde{\partial}_{l_n}\zeta_n+(\tilde{\partial}_{l_n}\zeta_n) A(\nabla_y\chi+e)v'_0+l_n(\tilde{\partial}_{l_n}\zeta_n) Ae\chi v''_0\vspace{3pt}\\
&=\int_{\R}-\kern-10pt\int_{\T^d}(\tilde{\partial}_{l_n}\zeta_n) A\tilde{\partial}_{l_n}\zeta_n -eA(\nabla_y\chi+e)v''_0\zeta_n +l_n(e\partial_{\xi}\zeta_n) Ae\chi v''_0- \nabla_y\cdot(Ae\chi)v''_0\zeta_n,\\
\end{split}
\end{equation*}
$\int_{\R}-\kern-9pt\int_{\T^d} c\zeta_n\partial_{\xi}v_n = \int_{\R}-\kern-9pt\int_{\T^d}cv_0'\zeta_n+l_nc\chi v_0''\zeta_n$, and $\int_{\R}-\kern-9pt\int_{\T^d}\beta v_n\zeta_n=\int_{\R}-\kern-9pt\int_{\T^d}\beta\zeta_n^2+\beta v_0\zeta_n+l_n\beta\chi v_0'\zeta_n$. For convenience, set 
$$I_1:=eA(\nabla_y\chi+e)+\nabla_y\cdot(Ae\chi)-A_0,$$ 
$$I_2:=(e\partial_{\xi}\zeta_n) Ae\chi v''_0-c\chi v''_0\zeta_n+\beta\chi v'_0\zeta_n,$$ 
and 
$$I_3:=eAe\chi v'''_0+c\chi v''_0-\beta\chi v'_0.$$
It is clear that 
$$-\kern-10pt\int_{\T^d} I_1=0,\quad  \left| \int_{\R}-\kern-10pt\int_{\T^d} I_2\right |\leq M_1,\quad  \int_{\R}-\kern-10pt\int_{\T^d}I_3^2\leq M_2, $$ 
for some constants $M_1>0,M_2>0$ independent of $n$ and $g$.  Since $M_{c,0}(v_0)=\bar{g}$, we have 
$$ \int_{\R}-\kern-10pt\int_{\T^d}(\tilde{\partial}_{l_n}\zeta_n) A\tilde{\partial}_{l_n}\zeta_n+\beta\zeta_n^2=  \displaystyle \int_{\R}-\kern-10pt\int_{\T^d}(\bar{g}-g+I_1v_0'')\zeta_n-l_nI_2.$$
Notice that $-\kern-9pt\int_{\T^d}(\bar{g}-g+I_1v_0'')=0$. It follows from the Poincar\'e  inequality that 
$$\baa{rcl}
\displaystyle\int_{\R}-\kern-10pt\int_{\T^d}(\tilde{\partial}_{l_{n}}\zeta_{n}) A\tilde{\partial}_{l_{n}}\zeta_{n}+\beta\zeta_{n}^{2} & \leq  &\displaystyle\int_{\R}-\kern-10pt\int_{\T^d}(\bar{g}-g+I_1v_0'')\zeta_n+|l_n| M_1\vspace{5pt}\\
& = & \displaystyle\int_{\R}-\kern-10pt\int_{\T^d}(\bar{g}-g+I_1v_0'')(\zeta_n-\overline{\zeta}_n)+|l_n|M_1\vspace{5pt}\\
& \leq & \displaystyle M_3\|\nabla_y\zeta_n\|_{L^2(\R\times \T^d)} +|l_n|M_1,\eaa$$
where $M_3>0$ is a constant independent of $n$ and $g$. Thanks to \eqref{uni-elliptic} and the fact $\beta>0$, this in particular implies that  $\zeta_n\to0$ in $L^2(\R\times \T^d)$  as $n\to+\infty$ uniformly for $g\in  B_M$. 

It remains to show $\partial_{\xi}\zeta_n\to 0$ in $L^2(\R\times \T^d)$ as $n\to\infty$ uniformly for $g\in B_M$. 
To do so, we define the symmetric difference quotient of $\zeta_n$ in the $\xi$-direction as follows
\begin{equation*}
D_hv_n(\xi,y):=\frac{v_n(\xi+h,y)-v_n(\xi-h,y)}{2h}\in H^1(\R\times\T^d).
\end{equation*} 
Notice that 
\begin{equation*}
\begin{split}
M_{c,L_n}(\zeta_n) &=g-cv'_0+\beta v_0-\big(eA(\nabla_y\chi+e)+\nabla_y\cdot(Ae\chi)\big)v''_0-l_nI_3\\
&=g-\bar{g}-I_1v''_0-l_nI_3.
\end{split}
\end{equation*} 
Integrating this equation against $D_h\zeta_n$ gives 
$$\displaystyle\left|c\int_{\R}-\kern-10pt\int_{\T^d}(\partial_{\xi} \zeta_n)^2\right| \leq \displaystyle \left|\int_{\R}-\kern-10pt\int_{\T^d}(g-\bar{g}-I_1v''_0)\partial_{\xi}\zeta_n \right|+|l_n|\left|\int_{\R}-\kern-10pt\int_{\T^d}I_3\partial_{\xi}\zeta_n\right|.$$
Since $-\kern-9pt\int_{\T^d} \partial_{\xi}(g-\bar{g}-I_1v''_0)=0$, it follows from the Cauchy-Schwarz inequality and the Poincar\'e inequality that 
$$\baa{rcl}
\displaystyle \left|\frac{c}{2}\int_{\R}-\kern-10pt\int_{\T^d}(\partial_{\xi} \zeta_n)^2\right| & \leq & \displaystyle \left|\int_{\R}-\kern-10pt\int_{\T^d}(g-\bar{g}-I_1v''_0)\partial_{\xi}\zeta_n \right|+\frac{l_n^2}{2c}\int_{\R}-\kern-10pt\int_{\T^d}I_3^2\vspace{5pt}\\
& \leq & \displaystyle \left|\int_{\R}-\kern-10pt\int_{\T^d}\partial_{\xi}(g-\bar{g}-I_1v''_0)\zeta_n \right|+\frac{l_n^2}{2c}M_3\vspace{5pt}\\
& = & \displaystyle \left|\int_{\R}-\kern-10pt\int_{\T^d}\partial_{\xi}(g-\bar{g}-I_1v''_0)(\zeta_n-\overline{\zeta}_n) \right|+\frac{l_n^2}{2c}M_3\vspace{5pt}\\
& \leq  & \displaystyle M_4\|\nabla_y\zeta_n\|_{L^2(\R\times \T^d)}+\frac{l_n^2}{2c}M_3,
\eaa$$
where  $M_4>0$ is independent of $n$ and $g$.  This implies that  $\partial_{\xi}\zeta_n\to 0$ as $n\to+\infty$ uniformly for $g\in  B_M$.  
Combining the above, we obtain the uniform convergence of $v_n$ to $v_0$ in  $H^1(\R\times \T^d)$ with respect to $g\in B_M$. 
This ends the proof of Lemma \ref{continuem}.
\end{proof}

Now, we introduce the operator for which the implicit function theorem is employed.  
Recall that $\phi_0$ is the traveling wave of \eqref{onehom} connecting $0$ and $1$ with speed $c_0>0$. Since $\bar{f}'(0)<0$ and $\bar{f}'(1)<0$,  
$\phi_0(\xi)$ approaches $0$ and $1$ exponentially fast as $\xi\to +\infty$ and $\xi\to -\infty$, respectively. This implies that $\phi_0\in L^2([0,+\infty))$, $1-\phi_0\in L^2((-\infty,0])$. Moreover, since $\bar{f}\in C^1(\R)$, by elliptic regularity, we have $\phi'_0\in H^2(\R)$. 
For any $v\in H^1(\R\times\T^d)$, $c>0$ and $l\in\R$, set
\begin{equation*}
\begin{aligned}
K(v,c,l)(\xi,y)=&\big(eA(y)(\nabla_{y}\chi(y)+e)+\nabla_{y}\cdot(A(y)e\chi(y))\big)\phi''_{0}(\xi)+l(eA(y)e\chi(y))\phi'''_{0}(\xi) \vspace{3pt}\\
&+c(\phi'_{0}(y)+l\chi(y)\phi''_{0}(\xi))+f(y,v(\xi,y)+\phi_{0}(\xi)+l\chi(y)\phi'_{0}(\xi))+\beta v(\xi,y).
\end{aligned}
\end{equation*}
It is straightforward to check that $K(v,c,l)\in L^2(\R\times\T^d)$. Define $G(v,c,l)=(G_1,G_2)(v,c,l)$ for $(v,c,l)\in H^1(\R\times\T^d)\times (0,+\infty)\times \R$  by 
\be\label{defG}\left\{\baa{rcl}
G_1(v,c,l) & = & \left\{\baa{ll}
v+M_{c,l}^{-1}\big(K(v,c,l)\big) & \hbox{if }l\neq0,\vspace{3pt}\\
v+M_{c,0}^{-1}\big(\overline{K(v,c,0)}\big) & \hbox{if }l=0,\eaa\right.\vspace{3pt}\\
G_2(v,c,l) & = & \displaystyle\int_{\R^+\times \T^d}\big(\phi_{0}(\xi)+v(\xi,y)+l\chi(y)\phi_{0}'(\xi)\big)^2-\int_{\R^+}\phi_0^2,\eaa\right.
\ee
where $\overline{K(v,c,0)}\in L^2(\R)$ is defined as $\bar{f}$ in~\eqref{average}. Thanks to \eqref{inverse-1} and \eqref{inverse-2}, we have $G$: $H^1(\R\times\T^d)\times(0,+\infty)\times\R \to H^1(\R\times\T^d)\times\R$. It is easily seen that $G(0,c_0,0)=(0,0)$. Moreover, by the parabolic regularity, 
a pair $(\phi_l,c^*_l)\in\big(\phi_0+H^1(\R\times\T^d)\big)\times(0,+\infty)$ solves the problem \eqref{pulsolu}-\eqref{con-limit} for $l\neq 0$ with the normalization 
$\int_{\R^+\times \T^d} \phi_l^2=\int_{\R^+} \phi_0^2$ if and only if $G(\phi_l-\phi_0-l\chi\phi_0',c^*_l,l)=(0,0)$.

On the other hand, by arguments similar to those used in the proof of \cite[Lemma 3.4]{dhz1}, one can conclude that the operator $G$ is continuous with respect to $(v,c,l)$ and it is continuously Fr\'echet differentiable with respect to $(v,c)$.  Furthermore, the derivative $\partial_{(v,c)}G(0,c_0,0): H^1(\R\times\T^d)\times\R\to H^1(\R\times\T^d)\times\R $ is invertible.

Finally, applying the implicit function theorem on the function $G: H^1(\R\times\T^d) \times (0,+\infty)\times\R\to H^1(\R\times\T^d) \times\R$, we obtain some $l^*:=l^*(e)>0$ and a unique continuous mapping $h(l)=(v_l,c^*_l)$ for  $l\in(0,l^*)$ such that $h(0)=(0,c_0)$, $G(h(l),l)=(0,0)$ and $h(l)\to h(0)$ as $l\to 0^+$. Letting $\phi_l(\xi,y)=\phi_0(\xi)+v_l(\xi,y)+l\chi(y)\phi_0'(\xi)$ for $(\xi,y)\in\R\times \T^d$, 
it follows that $(\phi_l,c^*_l)$ is a solution of \eqref{pulsolu}-\eqref{con-limit}, and that  $\phi_l-\phi_0\to0$ in $H^1(\R\times\T^d)$ as $l\to0^+$. 
Furthermore, by the assumption (A3) and the proof of Lemma \ref{exp-decay-hd-1}, any non-stationary pulsating wave of \eqref{eql} approaches the limiting states $0$ and $1$ exponentially fast. This together with the uniqueness of wave speed stated in Theorem \ref{th:tw} implies the 
the uniqueness of $(\phi_l,c^*_l)$ (we refer the reader to the end of the proof of \cite[Theorem 1.2]{dhz1} for more details). The proof of Theorem \ref{thhomo} is thus complete.


\begin{thebibliography}{AAA}


\bibitem{ag} M. Alfaro, T. Giletti, {\it Varying the direction of propagation in reaction-diffusion equations in periodic media}, Netw. Heterog. Media 11 (2016), 369-393.

\bibitem{abc}  N. D. Alikakos, P. W. Bates, X. Chen, {\it Periodic traveling waves and locating oscillating patterns in multidimensional domains}, Trans. Amer. Math. Soc.,  351 (1999), 2777-2805.

\bibitem{aw} D. G.  Aronson, H. F. Weinberger,  {\it Multidimensional nonlinear diffusion arising in population genetics}, Adv. Math.  30 (1978), 33-76.


\bibitem{bh1} H. Berestycki, F. Hamel, {\it Front propagation in periodic excitable media}, Commun. Pure Appl. Math. 55 (2002), no. 8, 949-1032.

\bibitem{bh2} H. Berestycki, F. Hamel, {\it Generalized transition waves and their properties}, Commun. Pure Appl. Math. 65 (2012), 592-648.


\bibitem{contri} B. Contri, {\it Pulsating fronts for bistable on average reaction-diffusion equations in a time periodic environment}, J. Math. Anal. Appl., 437 (2016), 90-132.


\bibitem{dg} W. Ding, T. Giletti, {\it Admissible speeds in spatially periodic bistable reaction-diffusion equations}, Adv.  Math. 389 (2021), 107889.

\bibitem{dhz2} W. Ding, F. Hamel, X.-Q.  Zhao, {\it Transition fronts for periodic bistable reaction-diffusion equations}, Calc. Var. Part. Diff. Equations 54 (2015), 2517-2551.

\bibitem{dhz1} W. Ding, F. Hamel, X.-Q.  Zhao, {\it Bistable pulsating fronts for reaction-diffusion equations in a periodic habitat}, Indiana Univ. Math. J. 66 (2017), 1189-1265.


\bibitem{dl} W. Ding, X. Liang, {\it Sign of the pulsating wave speed for the bistable competition–diffusion system in a periodic habitat}, Math. Ann. (2022), to appear.


\bibitem{dgm} A. Ducrot, T. Giletti, H. Matano, {\it Existence and convergence to a propagating terrace in one-dimensional reaction-diffusion equations}, Trans. Amer. Math. Soc. 366 (2014), 5541-5566.

\bibitem{ducrot} A. Ducrot, {\it A multi-dimensional bistable nonlinear diffusion equation in a periodic medium}, Math. Ann. 366 (2016), no. 1-2, 783-818.


\bibitem{fz} J. Fang,  X.-Q.  Zhao,  {\it Bistable traveling waves for monotone semiflows with applications}, J. Eur. Math. Soc. 17 (2015), 2243-2288.

\bibitem{fm1} P. C. Fife,  J. B. McLeod,  {\it The approach of solutions of nonlinear diffusion equations to travelling front solutions},  Arch. Ration. Mech. Anal. 65 (1977), 335-361.


\bibitem{guo} H. Guo, {\it Propagating speeds of bistable transition fronts in spatially periodic media}, Calc. Var. Part. Diff. Equations  57 (2018), 1-37.


\bibitem{gr} T. Giletti, L. Rossi, {\it Pulsating solutions for multidimensional bistable and multistable equations}, Math. Ann. 378(3–4) (2020) 1555–1611.


\bibitem{h} F. Hamel,  {\it Qualitative properties of monostable pulsating fronts: exponential decay and monotonicity}, J. Math. Pures Appl. 89 (2008), no. 4, 355-399.


\bibitem{m} H. Matano, {\it Traveling waves in spatially random media}, RIMS Kokyuroku 1337 (2003), 1-9.

\bibitem{he2} S. Heinze, {\it Wave solutions to reaction-diffusion systems in perforated domains}, Z. Anal. Anwendungen {\bf 20} (2001), 661-670.


\bibitem{nr} J. Nolen, L. Ryzhik, {\it Traveling waves in a one-dimensional heterogeneous medium}, Ann. Inst. H.~Poincar\'e, Analyse Non Lin\'eaire 26 (2009), 1021-1047.

\bibitem{po2} Pol{\'a}{\v c}ik, P.,  {\it Propagating terraces and the dynamics of front-like solutions of reaction-diffusion equations on $\R$}, Mem. Amer. Math. Soc. 264 (2020), no. 1278, v+87 pp.


\bibitem{xin} J. Xin, {\it  Existence and stability of traveling waves in periodic media governed by a bistable nonlinearity}, J. Dyn. Diff. Eq. 3 (1991), 541-573.

\bibitem{xin2} J. Xin, {\it  Existence and nonexistence of traveling waves and reaction-diffusion front propagation in periodic media}, J. Statist. Phys. 73 (1993), no. 5-6, 893-926.

\bibitem{z} A. Zlato{\v{s}}, {\it Existence and non-existence of transition fronts for bistable and ignition reactions}, 
Ann. Inst. H.~Poincar\'e, Analyse Non Lin\'eaire 34 (2017), 1687-1705. 

\bibitem{zz21} Y. Zhang, X.-Q. Zhao, {\it Uniqueness and Stability of Bistable Waves for Monotone Semiflows}, Proc. Amer. Math. Soc., 149 (2021), 4287-4302.



\end{thebibliography}
\end{document}